\documentclass[11pt]{article}

\RequirePackage[l2tabu, orthodox]{nag}
\usepackage[T1]{fontenc}
\usepackage[utf8]{inputenc}
\usepackage{lmodern}
\usepackage[english]{babel}
\usepackage{color} 
\usepackage[usenames,dvipsnames]{xcolor}
\usepackage{amsrefs,amsthm,amssymb,amsmath}
\usepackage[pdftex]{graphicx}
\usepackage{enumerate}
\usepackage{url}

\newcommand{\emr}[1]{{#1}}   

\def\comment#1{}

\newcommand{\R}{\mathbb{R}}
\newcommand{\N}{\mathbb{N}}
\newcommand{\Z}{\mathbb{Z}}
\newcommand{\Q}{\mathbb{Q}}
\renewcommand{\P}{\mathsf{P}^1}
\newcommand{\T}{\mathbb{S}^1}
\newcommand{\Diff}{\mathsf{Diff}}
\newcommand{\Homeo}{\mathsf{Homeo}}

\newcommand{\cP}{\mathcal{P}}

\renewcommand{\a}{\alpha}

\newcommand{\PSL}{\mathsf{PSL}}
\newcommand{\SL}{\mathsf{SL}}
\newcommand{\BS}{\mathsf{BS}}
\newcommand{\mP}{\mathsf{P}}
\newcommand{\cH}{\mathcal{H}}
\newcommand{\Fix}{\mathsf{Fix}}
\newcommand{\0}{\mathtt{0}}
\newcommand{\1}{\mathtt{1}}

\DeclareMathOperator{\Tr}{\mathsf{Tr}}

\newtheorem{thm}{Theorem}[section]
\newtheorem*{thm*}{Theorem}
\newtheorem*{claim}{Claim}

\newtheorem{lem}[thm]{Lemma}
\newtheorem{prop}[thm]{Proposition}
\newtheorem{cor}[thm]{Corollary}
\theoremstyle{definition}
\newtheorem{dfn}[thm]{Definition}
\theoremstyle{remark}
\newtheorem{rem}[thm]{Remark}
\newtheorem{ex}[thm]{Example}
\newtheorem{ques}[thm]{Question}

\usepackage[margin=2.5cm]{geometry}

\usepackage{indentfirst}
\usepackage{microtype}

\usepackage[colorlinks=true,linkcolor=blue,citecolor=magenta]{hyperref}

\title{Hyperbolicity as an obstruction to smoothability for one-dimensional actions}
\date{}
\author{Christian Bonatti\thanks{CNRS.} \thanks{Institut de Math\'ematiques de Bourgogne, Universit\'e Bourgogne Franche-Comt\'e, CNRS UMR 5584,
9 av.~Alain Savary,
21000 Dijon, France.} \and Yash Lodha\thanks{EPFL SB MATH EGG 
MA C3 584 (Batiment MA) 
Station 8 
CH-1015 Lausanne, CH-1015, Switzerland.} \and Michele Triestino\footnotemark[2]}
\begin{document}
\maketitle

\begin{abstract}
Ghys and Sergiescu proved in the $80$s that Thompson's group $T$, and hence $F$, admits actions by $C^{\infty}$ diffeomorphisms of the circle .
They proved that the standard actions of these groups are topologically conjugate to a group of $C^\infty$ diffeomorphisms. 
Monod defined a family of groups of piecewise projective homeomorphisms,
and Lodha-Moore defined finitely presentable groups of piecewise projective homeomorphisms. These groups are of particular interest because they are nonamenable and contain no free subgroup.
In contrast to the result of Ghys-Sergiescu, we prove that the groups of Monod and Lodha-Moore are not topologically conjugate to a group of $C^1$ diffeomorphisms.

Furthermore, we show that the group of Lodha-Moore has no nonabelian $C^1$~action on the interval.
We also show that many Monod's groups $H(A)$, for instance when $A$ is such that $\PSL(2,A)$ contains a rational homothety $x\mapsto \tfrac{p}{q}x$,
do not admit a $C^1$~action on the interval. 
The obstruction comes from the existence of hyperbolic fixed points for $C^1$ actions. 
With slightly different techniques, we also show that some groups of piecewise affine homeomorphisms of the interval or the circle are not smoothable.
\footnote{\textbf{MSC\textup{2010}:} Primary 37C85, 57M60. Secondary 43A07, 37D40, 37E05.
}

\end{abstract}

\tableofcontents

\section{Introduction}

A few examples are known of groups that admit no sufficiently smooth action on a one-dimensional manifold. Following the direction of Zimmer program,
typical examples come from lattices in higher rank Lie groups \cite{burger-monod,witte,ghys}, or more generally from groups with Kazhdan's property $(T)$
\cite{navas(T),navas-rel}. Other interesting examples appear in \cite{forcing,calegari,parwani,navas-locind,baik-kim-koberda}.

In this work we address the problem of the existence of \emph{smooth actions of groups of piecewise projective homeomorphisms of the real line}. Our
principal interest comes from the existence of groups of this kind which are negative solutions to the so-called Day-von Neumann problem, as shown
by Monod and Lodha-Moore \cite{monod,LodhaMoore}. On the other hand, partially motivated by his work on Kazhdan groups acting on the circle,
 Navas raised the problem to find obstructions for a group of piecewise linear homeomorphisms of the interval to admit smooth actions 
 (\textit{cf.}~\cite{navas-locind,hyperbolic}). With this work, we illustrate relatively elementary tools which apply to a large variety of 
 examples of such groups. Our techniques rely on some classical facts on one-dimensional dynamics and the recent work by Bonatti, 
 Navas, Rivas and Monteverde on actions of abelian-by-cyclic groups \cite{hyperbolic}.   

A classical obstruction to have $C^1$ actions on the interval is Thurston's Stability Theorem \cite{Th}:
a group of $C^1$ diffeomorphisms of the interval is \emph{locally indicable}, 
namely every finitely generated subgroup has a nontrivial morphism to $\Z$.
This obstruction does not apply in our setting: 
the group of piecewise projective homeomorphisms of the real line is locally indicable.
 Therefore our results exhibit new examples of locally indicable groups that have no $C^1$ action on the interval. 

\smallskip

As an appetizer, even before introducing the notions and definitions which are necessary for presenting our main results, 
we start with two results whose statements are very easy to understand, and which illustrate the spirit of the paper. 
Fix $\lambda>1$ and consider:\begin{itemize}
\item the linear map $f_\lambda\colon \R\to\R$ defined as $x\mapsto \lambda x$,
\item the map $h_\lambda\colon \R\to \R$ defined as
$h_\lambda(x)=\begin{cases}
x&\text{if }x\le 0,\\
\lambda x&\text{if }x>0,
\end{cases}$
\item the translation $g\colon x\mapsto x+1$.
\end{itemize}
Let $G_\lambda$ be the subgroup $\langle f_\lambda,g,h_\lambda\rangle\subset \Homeo_+(\R)$.

\begin{thm}\label{t:main1}
For any $\lambda>1$ which is rational (in formula: $\lambda\in \Q\cap (1,+\infty)$)
and any morphism $\rho\colon G_\lambda \to \Diff^1_+([0,1])$ one has :
\begin{center}
the commutator $[g,h_\lambda gh_\lambda^{-1}]$ belongs to the kernel of $\rho$. 
\end{center}
In particular $\rho$ cannot be injective.

The same holds for any morphism $\varphi\colon G_\lambda\to \Diff^1_+(\T)$, where $\T$ is the circle.
\end{thm}

In fact, we get the stronger conclusion that for any representation $\rho\colon G_\lambda \to \Diff^1_+([0,1])$, the image $\rho(G_\lambda)$ is a metabelian group (that is, a solvable group with abelian derived subgroup).

\emr{The same occurs for a more general class of algebraic numbers, that we call \emph{Galois hyperbolic} (see Definition~\ref{HyperbolicAlgebraic} and Theorem~\ref{t:main1Hyp}).}
We do not know if the same occurs for $\lambda>1$ \emr{not Galois hyperbolic (see Remark~\ref{couterex})}. Nevertheless, consider the natural realization  $\rho_0\colon G_\lambda\to \Homeo_+(\T)$ defined as follows:
\begin{itemize}
 \item one considers $\T$ as being the projective space $\R \P$,
 \item $\rho_0(f_\lambda)$ acts on $\T$ as the projective action of the matrix 
$\begin{pmatrix}
    \lambda&0\\
    0&1
   \end{pmatrix}$,

\item $\rho_0(g)$ acts on $\T$ as the projective action of the matrix 
 $\begin{pmatrix}
    1&1\\
    0&1
\end{pmatrix}$,

\item $\rho_0(h_\lambda)$ coincides with $\rho_0(f_\lambda)$ on the half circle $[0, +\infty]$ and  with the identity map on the half circle $[-\infty, 0]$. 
\end{itemize}
\begin{thm}\label{t:main2} Fix an arbitrary real number $\lambda>1$.
With the notation as above, it does not exist any homeomorphism $\phi\colon \T\to \T$ so that 
$\phi\rho_0(f_\lambda)\phi^{-1}$, $\phi\rho_0(h_\lambda)\phi^{-1}$ and $\phi\rho_0(g)\phi^{-1}$ belong to $\Diff_+^1(\T)$. 
\end{thm}

In other words, the natural action of $G_\lambda$ on $\T$ is not \emph{smoothable}, and furthermore, \emr{if $\lambda>1$ is Galois hyperbolic}, then every $C^1$ action of $G_\lambda$ on the 
circle or the interval are (non-faithful) metabelian actions. 

For more precise statements, see Theorems  \ref{t:affine} and \ref{t:key}.

\smallskip

The paper is organized as follows. In Section \ref{s:Defs} we introduce the basic objects and fix some notation. In Section \ref{s:Mech} we roughly explain the different strategies that we develop in this work, showing which are the main applications.
In Section \ref{s:Motiv} we illustrate the main motivation of our work, which is the recent construction by Monod of nonamenable groups without free subgroups. In Section \ref{s:Monod} we study the $C^1$ actions of Monod's groups and the finitely presentable group defined by Lodha-Moore. Section \ref{s:Aff} contains the main part of this work, namely the study of $C^1$ actions of the groups $G_\lambda$ introduced above. Finally, in Section \ref{s:C2} we use different techniques that work in $C^2$ regularity.

\section{Some definitions and notation}
\label{s:Defs}

\begin{dfn}
Let $M$ be a manifold and $\Homeo(M)$ the group of homeomorphisms of $M$. 
A subgroup $G\subset \Homeo(M)$ is \emph{$C^r$-smoothable} ($r\ge 1$) if it is conjugate in $\Homeo(M)$ to a subgroup in $\Diff^r(M)$, the group of $C^r$ diffeomorphisms of $M$.
\end{dfn}

\begin{rem}\label{r:abstract}
Even if a certain subgroup $G\subset \Homeo(M)$ is not $C^r$-smoothable, it is still possible that the group $G$, \emph{as abstract group}, admits $C^r$ actions on the manifold $M$.
\end{rem}

Throughout this work we shall only be concerned by one-dimensional manifolds. 
We restrict our discussion to orientation-preserving homeomorphisms, which form a subgroup $\Homeo_+(M)$ of index two in $\Homeo(M)$.
We will not make much distinction between the groups $\Homeo_+(\R)$ and $\Homeo_+([0,1])$. 
Notice however that the groups $\Diff_+^r(\R)$ and $\Diff^r_+([0,1])$ are different and for this reason we sometimes identify the 
interval $[0,1]$ to the compactified real line $[-\infty,+\infty]$. Choosing the affine chart $t\mapsto [t:1]$, 
we consider $\R$ as the affine line in the projective space $\R\P\cong\R\cup \{\infty\}$, which is topologically the circle $\T$. 
The group $\Homeo_+(\R)$ can be identified to a subgroup of $\Homeo_+(\T)$, for instance as the stabiliser of the point $\infty$ of $\T\cong\R\P$.

The \emph{projective special linear group} $\PSL(2,\R)= \SL(2,\R)/\{\pm id\}$ naturally acts on the projective real line $\R\P$ by Möbius transformations: 
from now on, we shall always suppose that $\PSL(2,\R)$ acts on the circle in this way.

\begin{dfn}
A circle homeomorphism $h\in\Homeo_+(\R\P)$ is \emph{piecewise projective} if there exists a finite partition 
$\R\P=I_1\cup\ldots \cup I_\ell$ of the circle into intervals, such that every restriction $h\vert_{I_k}$, $k=1,\ldots,\ell$, 
coincides with the restriction of a M\"obius transformation.

A \emph{breakpoint} of $h$ is a point $b\in\R\P$ such that the restriction of $h$ to any neighbourhood of $b$ does not coincide 
with the restriction of a M\"obius transformation. 

The group of all orientation-preserving piecewise projective homeomorphisms of the circle is denoted by $\mP\mP_+(\R\P)$. 
Similarly, we define the group of piecewise projective homeomorphisms of the real line $\mP\mP_+(\R)$, identifying it to the stabiliser of $\infty$ inside $\mP\mP_+(\R\P)$.
\end{dfn}

We recall that a fixed point $p\in\R$ for a diffeomorphism $f\in\Diff^1_+(\R)$ is a \emph{hyperbolic fixed point} if $f$ 
has derivative at $p$ which is not $1$. We shall say that a subgroup $G\subset \Diff^1_+(\R)$ has hyperbolic fixed points if there exists 
an element $f\in G$ with hyperbolic fixed points.
This notion is related to the notion of \emph{hyperbolic elements} in $\PSL(2,\R)$.
A nontrivial projective transformation  in $\PSL(2,\R)$ has at most two fixed points.
If it has exactly two fixed points, it is called \emph{hyperbolic}, and if it has only one fixed point it is called \emph{parabolic}.  
A matrix $M$ in $\mathsf{SL}(2,\R)$ is hyperbolic if $|\Tr(M)|>2$,
parabolic if $|\Tr(M)|=2$ and elliptic if $|\Tr(M)|<2$.
Then the corresponding projective transformation is respectively hyperbolic, parabolic and elliptic.

Given a subgroup $\Gamma\subset \PSL(2,\R)$, we say that a real $r\in \mathbb{R}$ is a \emph{hyperbolic fixed point} for $\Gamma$
if there is a $\gamma\in \Gamma$ such that $\gamma$ is hyperbolic and $\gamma(r)=r$.
Similarly, we define the notion of a \emph{parabolic fixed point} for $\Gamma$.
We consider the sets $\cH_\Gamma$ and $\cP_\Gamma$ of \emph{hyperbolic fixed points} and \emph{parabolic fixed points} of elements of $\Gamma$, respectively.
When $\Gamma=\PSL(2,A)= \SL(2,A)/\{\pm Id\}$, for some subring $A\subset \R$, we simply write $\cH_A$ and $\cP_A$. 
Here $\SL(2,A)$ is the group of invertible $(2\times 2)$-matrices with determinant~$1$ and coefficients in $A$.

\smallskip

\emr{Let $\lambda\in \R$ be an algebraic real number of degree $d$ over $\Q$, and let $p_\lambda(t)=\frac{\a_0}{\a_d}+\frac{\a_1}{\a_d}t+\ldots +\frac{\a_{d-1}}{\a_d}t^{d-1}+t^d$, $\a_j\in \Z$, denote the associated minimal polynomial. The field $\Q(\lambda)$ is a $\Q$-vector space of dimension $d$, for which we fix $\{1,\lambda,\ldots,\lambda^{d-1}\}$ as preferred basis. With respect to this basis, multiplication by $\lambda$ on $\Q(d)$ is represented by the matrix
\begin{equation}\label{Companion}
C_\lambda=\left (\begin{array}{ccc|c}
0&\cdots&0 & -\a_0/\a_d \\
\hline 
&&& -\a_1/\a_d \\
&I_{d-1}& & \vdots \\
&&& -\a_{d-1}/\a_d
\end{array}\right )
\end{equation}
which is commonly named the \emph{Frobenius companion matrix} of $\lambda$. When $\lambda\neq 0$ then $\a_0\neq 0$, so that $C_\lambda$ is an invertible $d\times d$ matrix with rational coefficients. The minimal polynomial of $C_\lambda$ is exactly $p_\lambda$, so the eigenvalues of $C_\lambda$ are exactly the \emph{Galois conjugates} of $\lambda$, that is, all (complex) roots of $p_\lambda$.
\begin{dfn}\label{HyperbolicAlgebraic}
	A nonzero real number $\lambda\in \R$ is \emph{Galois hyperbolic} if it is algebraic and the companion matrix $C_\lambda$ has no eigenvalue of absolute value $1$. Equivalently, this means that all the Galois conjugates of $\lambda$ do not have absolute value $1$.
\end{dfn}
For instance, any rational $\lambda\neq 0,\pm 1$ is Galois hyperbolic, as well as any quadratic integer $\sqrt{m}\neq 0,1$, $m\in \N$. 
However not every real number is Galois hyperbolic. As an explicit nontrivial example \cite[\S~5]{hyperbolic}, the polynomial $p(t)=1+4t+4t^2+4t^3+t^4$ is irreducible over $\Q$, has two positive real roots $\lambda$, $1/\lambda$
and two roots of absolute value $1$. Hence $\lambda,1/\lambda$ are not Galois hyperbolic.}

\emr{Theorem~\ref{t:main1} holds for
this more general class of numbers.}
\begin{thm}\label{t:main1Hyp}
	For any Galois hyperbolic $\lambda>1$
	and any morphism $\rho\colon G_\lambda \to \Diff^1_+([0,1])$ one has :
	\begin{center}
		the commutator $[g,h_\lambda gh_\lambda^{-1}]$ belongs to the kernel of $\rho$. 
	\end{center}
	In particular $\rho$ cannot be injective.
	
	The same holds for any morphism $\varphi\colon G_\lambda\to \Diff^1_+(\T)$, where $\T$ is the circle.
\end{thm}

\section{The mechanisms}
\label{s:Mech}
The aim of this work is to present three different techniques which provide a variety of examples of non-smoothable groups in $\mP\mP_+(\R)$. 
The three techniques rely on the \emph{rigid hyperbolicity} of the actions: there are subgroups $G\subset \Diff_+^1(\R)$ such that, 
no matter how one (topologically) conjugates them inside $\Diff^1_+(\R)$, will always have hyperbolic fixed points. 

More precisely, suppose that in $G\subset \Diff_+^1(\R)$ there is an element $f$ having a hyperbolic fixed point $p\in\R$. 
Consider another subgroup $\widetilde G\subset \Diff_+^1(\R)$ to which $G$ is topologically conjugate by some homeomorphism 
$\phi$: $\phi G\phi^{-1}=\widetilde G$. The point $\phi(p)$ is a fixed point for $\phi f \phi^{-1}$, 
but since $\phi$ is just a homeomorphism, we cannot ensure that it is a \emph{hyperbolic} fixed point. 
However, there are some topological mechanisms that guarantee hyperbolicity.

The first one is when there are \emph{linked pairs of fixed points} in $G$.
We now define this notion. 
Denote by $\Fix(g)$ the set of fixed points of a homeomorphism $g$.
A \emph{pair of successive fixed points} of $G$ is a pair $a,b\in \R$, $a<b$
such that there is an element $g\in G$ and $(a,b)$ is a connected component of $\R\setminus \Fix(g)$.
A \emph{linked pair of fixed points} consists of pairs $a,b$ and $c,d$ such that:
\begin{enumerate}[1)]
\item there are elements $f,g\in G$ such that $a,b$ is a pair of successive fixed points of $f$ and $c,d$ is a pair of successive fixed points of $g$;
\item either $\{a,b\}\cap (c,d)$ or $(a,b)\cap \{c,d\}$ is a point.
\end{enumerate}

In this case hyperbolicity is obtained by a probabilistic argument.
\emph{Some} element $h$ in the semigroup generated by $f$ and $g$ will have a hyperbolic fixed point \emph{somewhere}. 
This is the so-called Sacksteder's Theorem, in its version for $C^1$-pseudogroups \cite{DKNacta,navas-book}. This method applies to 
large groups of piecewise projective homeomorphisms, as the \emph{Monod's groups} (see Definition \ref{d:Monod}):
\begin{thm}\label{t:H(A)}
The following holds for Monod's groups $H(A)$ and $G(A)$.
\begin{enumerate}[(1)]
\item\label{i:Monod1} For any subring $A\subset \R$, Monod's groups $H(A)$ and $G(A)$ are not $C^1$-smoothable.
\item\label{i:Monod2} \emr{If $A$ contains $\sqrt{\lambda}^{\pm 1}$ 
for some Galois hyperbolic $\lambda>1$,} then there exists no injective morphism $\rho:H(A)\to \Diff_+^1([0,1])$.
\end{enumerate}
\end{thm}

\begin{rem}\emr{
	Condition \eqref{i:Monod2} on $A$ is equivalent to the fact that the group $\PSL(2,A)$ contains the homothety 
	\[f_\lambda\colon x\mapsto \lambda x,\]
	representing the matrix $\begin{pmatrix}
	\sqrt{\lambda} & 0 \\ 0 & \sqrt{\lambda}^{-1}
	\end{pmatrix}$.}
\end{rem}

The second one is when there is an \emph{exponential growth of orbits}. In this case we can ensure that a \emph{specific} point is always a hyperbolic fixed point. 
This applies for example to the dyadic affine group $\langle t\mapsto t+1, t\mapsto 2t\rangle$, which is isomorphic to the solvable Baumslag-Solitar group $\mathsf{BS}(1,2)$, 
as described in \cite{hyperbolic}. 
From this, it is easy to build examples of finitely generated groups in $\mP\mP_+(\R)$ which are not $C^1$-smoothable. This method applies to the finitely presentable \emph{Lodha-Moore group} (see \S~\ref{sc:LM}), for which we do not only prove that its action is not $C^1$-smoothable, but also that it has no nontrivial $C^1$ action on the interval:
\begin{thm}\label{t:LM}
Every morphism from the Lodha-Moore group $G_0$ to $\Diff_+^1([0,1])$ has abelian image.
\end{thm}

The third one relies on the nature of \emph{stabilisers}, and here we require that the regularity of the group $G$ is $C^2$.
If there exists a point $x\in\R$ such that the (right, for instance) germs of elements  $g\in G$ fixing $x$ define a group which is dense in $\R$, 
then we can use the \emph{Szekeres vector field} to have a  well-defined local differentiable structure, by means of which we ensure that the hyperbolic 
nature of a fixed point cannot change after topological conjugacy to another $C^2$ action. This method applies to examples of groups in $\mP\mP_+(\R)$ 
that are naturally in $\Diff^1_+(\R)$, e.g.~the group generated by Thompson's group $F$ (which is $C^1$ in $\mP\mP_+(\R)$) together with $t\mapsto t+\frac12$, for which we establish that their actions are not $C^2$-smoothable.
It also applies to the \emph{Thompson-Stein groups} $F(n_1,\ldots, n_k)$ and $T(n_1,\ldots, n_k)$ (see Definition \ref{d:TS}), extending a previous work by Liousse \cite{Liousse}:
\begin{thm}\label{t:F23}
The Thompson-Stein groups $F(n_1,\ldots, n_k)$, $k\ge 2$, are not $C^2$-smoothable.
\end{thm}

\begin{cor}\label{t:T23}~
\begin{enumerate}[(1)]
\item Assume $n_1=2$.
The Thompson-Stein groups $T(2,n_2,\ldots,n_k)$, $k\ge 2$, have no faithful $C^2$ action on $\T$.
\item Assume $n_1=2,n_2=3$. Every $C^2$ action of $T(2,3,n_3,\ldots,n_k)$ on $\T$ is trivial. This holds in particular for $T(2,3)$.
\end{enumerate}

\end{cor}

\section{Historical motivations}
\label{s:Motiv}
\subsection{Thompson's groups $F$ and $T$} In the 50s, Richard J.~Thompson introduced three groups $F$, $T$ and $V$, 
which have many nice properties (cf.~\cite{CFP}). These groups are \emph{finitely presented} and $[F,F]$, $T$, $V$ are \emph{simple}. 
They have been among the first known examples sharing these properties. 
Since only $F$ and $T$ act by \emph{homeomorphisms} on the circle, we restrict our attention to them.

\begin{dfn}
\emph{Thompson's group $T$} is the group of all piecewise linear homeomorphisms of the circle $\T\cong \R/\Z$ 
such that all derivatives are powers of $2$ and the breakpoints are dyadic rationals, i.e.~points of the form $p/2^q$, $p,q\in\N$. 
\emph{Thompson's group $F$} is the stabiliser of the point $0$ in $T$.
\end{dfn}

It has been proved by Ghys and Sergiescu in \cite{GS} that the piecewise linear action of $T$ (and hence of $F$) on $\T$ is \emph{$C^\infty$-smoothable}. 
On the other side, it is ``not difficult'' to find $C^\infty$ faithful actions (\textit{a priori} not topologically conjugate to the standard one) of Thompson's group. 

We recall \emph{Thurston's interpretation} of $T$ as a group of piecewise projective homeomorphisms of $\R\P$ (cf.~\cite{CFP}).
\begin{dfn}\label{ThurstonT}
$T$ is the group of piecewise $\PSL(2,\Z)$ homeomorphisms of $\R\P$ with breakpoints in $\cP_\Z$ (which is the set of rational numbers together with the point at infinity). 
$T$ is generated by $\PSL(2,\Z)$ and an additional element $c$ defined as
\[
c(t)=\begin{cases}
t&\text{if }t\in [\infty,0],\\
\frac{t}{1-t}&\text{if }0\leq t\leq\frac{1}{2},\\
3-\frac{1}{t}&\text{if }\frac{1}{2}\leq t\leq 1,\\
t+1&\text{if }t\in [1,\infty].
\end{cases}
\]
\end{dfn}
It is particularly striking that the element \emph{$c$ has continuous first derivative}. As the action of $\PSL(2,\Z)$ is even real-analytic, 
Thurston's interpretation gives a natural $C^1$-smoothing of $T$.\footnote{Another way of seeing this is that $C^1$ continuity follows from the choice of $\mathcal P_{\Z}$ for the set of breakpoints.}
In this model, $F$ is the group of piecewise $\PSL(2,\Z)$ homeomorphisms of $\R\P$ with breakpoints in $\cP_\Z$,
that also fix infinity.
So $F$ is the stabiliser of $\infty$ in $T$.
$F$ is generated by $t\mapsto t+1$ together with $c$ from above.
Recall that the group $\PSL(2,\Z)$ is isomorphic to the free product $\Z_2*\Z_3$, freely generated by the order two element $a: t\mapsto -\frac{1}{t}$ 
and the order three element $b: t\mapsto \frac{1}{1-t}$. 

\smallskip

Now we sketch a proof that $F$ admits a $C^{\infty}$ action inspired by \cite{chains} (see also \cite{fast}).
Note that this is weaker than proving it is $C^{\infty}$-smoothable,
which is a consequence of the theorem of Ghys-Sergiescu.

Given any homeomorphism $h:[0,1]\to [0,2]$, if we define the element
\[
\widetilde c(t)=\begin{cases}
t&\text{if }t\in [\infty,0],\\
h(t)&\text{if }t\in [0,1],\\
t+1&\text{if }t\in [1,\infty]
\end{cases}
\]
then the group generated by $t\to t+1$ and $\widetilde c$ is isomorphic to $F$. If we choose $h$ to be $C^\infty$, infinitely tangent to the identity at $0$ and to $t\mapsto t+1$ at $1$, then the modified element $\widetilde c$ is $C^\infty$. The algebraic properties of $F$ guarantee that the group generated by $t\to t+1$ and $\widetilde c$ is isomorphic to $F$.\footnote{To see this, first check that the relations of $F$ are satisfied and conclude using the property $F$ satisfies that every proper quotient is abelian.}
However, it is not guaranteed that one can choose $h$ and hence $\widetilde c$ such that the action of the group $\langle t\to t+1,\widetilde c\rangle$ is actually \emph{conjugate} to the standard action of $F$. 

A very important remark is that this strategy is morally possible because $0$ and $1$ \emph{are not hyperbolic fixed points} (they are parabolic). 
This allows to slow-down the dynamics near these points and make $c$ infinitely tangent to the identity. 
This feature already appeared in the work of Ghys and Sergiescu.
Hyperbolicity is a typical obstruction for such modifications in differentiable dynamics.

\subsection{One open problem: The Day-von Neumann problem for $\Diff^{2}_+(\R)$}
One of the main motivations for our work is understanding 
\emph{amenable groups} of diffeomorphisms of the circle. 
There are several equivalent definitions of amenability and an extensive literature on the topic (see~\cite{tullio} for an elementary introduction). 
We provide one definition:

\begin{dfn}
A discrete group $G$ is amenable if it admits a finitely additive, left translation invariant probability measure.
\end{dfn}

Here is an equivalent definition, \emph{\`a la} Krylov-Bogolyubov, which is more natural from the viewpoint of dynamical systems:
\begin{dfn}
A discrete group $G$ is amenable if every continuous action on a compact space has an invariant probability measure.
\end{dfn}

The class of amenable groups includes finite, abelian and solvable groups. 
Amenability is closed under extensions, products, direct unions and quotients.
Subgroups of amenable groups are amenable.
On the other hand, groups containing nonabelian free subgroups are nonamenable. 
The so-called \emph{Day-von Neumann problem} (popularized by Day in the 50s) is about the converse statement: \emph{does every nonamenable group contain a nonabelian free subgroup?}
If one restricts the question to \emph{linear groups}, then the well-known Tits alternative gives a positive answer: any non virtually solvable linear group contains nonabelian free subgroups.

The problem has been solved with negative answers and currently various negative solutions are known.
These include Tarski monsters, Burnside groups,
and Golod-Shafarevich groups. 
In this article we are interested in a particular class of such groups, discovered by Monod \cite{monod} and Lodha-Moore \cite{LodhaMoore}, which are subgroups of $\mP\mP_+(\R)$. 
Among them, there are examples that are in $\Diff_+^1(\R)$.
For instance, the group generated by $t\to t+\frac{1}{2}$ together with the element $c$ from Definition~\ref{ThurstonT} above provides such an example.

Interestingly, no negative solution to the Day-von Neumann problem is known among subgroups of $\Diff^2_+(\R)$.
Motivated by this question, in this work we prove (Theorems \ref{t:Firr} and \ref{t:Frat}) that the natural actions of these groups are not $C^2$-smoothable.
However, we have to stress that \textit{a priori} there could be smooth actions of such nonamenable groups that are not topologically conjugate to the standard actions (\textit{cf.}~Remark~\ref{r:abstract}).

\smallskip

The moral consequence of our results is that the Day-von Neumann problem in $\Diff_+^2(\R)$ is strictly harder than in 
$\Diff^1_+(\R)$. 
This is not so surprising, since there are important differences between $C^2$ and $C^1$ diffeomorphisms in one-dimensional dynamics.
We end this section by recalling a couple of tantalising longstanding open questions in this direction.

\begin{ques}
Is $F$ amenable?
\end{ques}

\begin{ques}
Does the Tits alternative hold for the group of real-analytic diffeomorphisms of the real line?
\end{ques}

\subsection{A second open problem: Higher rank behaviour}

\begin{dfn}\label{d:TS}
Let $1<n_1<\cdots<n_k$ be natural numbers such that the group $\Lambda=\langle n_i\rangle\subset \R_+^*$ is an abelian group of rank $k$. Denote by $A$ the ring $\Z[\frac1m]$, where $m$ is the least common multiple of the $n_i$'s.

\emph{Thompson-Stein's group} $T(n_1,\ldots,n_k)$ is the group of all piecewise linear homeomoprhisms of the circle $\T\cong\R/\Z$ such that all derivatives are in $\Lambda$ and the breakpoints are in $A$. \emph{Thompson-Stein's group} $F(n_1,\ldots,n_k)$ is the stabiliser of the point $0$ in $T(n_1,\ldots,n_k)$.
\end{dfn}

With the above definition, the group $T(2)$ is the classical Thompson's group $T$. It has been proved by Stein~\cite{stein} that these groups share many group-theoretical properties with the classical Thompson's groups, such as being finitely presentable (\textit{cf.}~\cite{bieri-strebel}).

However there are important differences from the dynamical viewpoint. In \cite{minakawa1,minakawa2}, Minakawa discovers that $\mathsf{PL}_+(\T)$ contains  ``exotic circles'', namely topological conjugates of $\mathsf{SO}(2)$ that are not one-parameter groups \emph{inside $\mathsf{PL}_+(\T)$}, in the sense that they are not $\mathsf{PL}$ conjugates of $\mathsf{SO}(2)$. 
In particular, Liousse shows in \cite{Liousse} that $T(n_1,\ldots,n_{k})$ contains an abelian group of rank $k-1$ that is contained in a topological conjugate of $\mathsf{SO}(2)$, but not in a $\mathsf{PL}$ conjugate of $\mathsf{SO}(2)$.
Whence Navas suggested the following:
\begin{ques}
Does $T(n_1,\ldots,n_k)$, $k\ge 2$, satisfy Kazhdan's property $(T)$?\footnote{We do not define property $(T)$ here, we refer the reader to \cite{kazhdan}.}
\end{ques}
On the other hand Navas proved in \cite{navas(T)} that the only groups of $C^{r}$ diffeomorphisms, $r> 3/2$, that have property $(T)$ are finite. Focusing our attention on one particular example, in \cite{Liousse}, Liousse proves, among other things, that every action of $T(2,3)$ on $\T$ by $C^9$ diffeomorphisms is trivial and with Corollary \ref{t:T23} we improve this result to $C^2$ regularity. 
It would be very interesting to prove that $T(2,3)$ has no $C^1$ action on the circle, as this would confirm that this group is a good candidate 
for finding an infinite Kazhdan group of circle homeomorphisms.

Naturally, there could be also good candidates among groups of piecewise projective homeomorphisms.
\section{Nonamenable groups of piecewise projective homeomorphisms}
\label{s:Monod}
\subsection{Monod's groups}

Generalizing a well-known result by Brin and Squier \cite{brin-squier}, 
Monod showed in \cite{monod} that $\mP\mP_+(\R)$ does not contain nonabelian free subgroups. One key feature is that given any $r\in \mathbb{R}$, the group of germs of elements in $\mP\mP_+(\R)$ fixing the point $r$
is isomorphic to the affine group.

\begin{dfn}[Monod's groups] \label{d:Monod} Let $A$ be a subring of $\R$. 
$G(A)$ is defined as the group of all piecewise $\PSL(2,A)$ homeomorphisms of the circle 
with breakpoints in $\cH_A$. The group $H(A)$ is the stabiliser of $\infty$ inside $G(A)$.
\end{dfn}

Observe that the groups $G(\R)$ and $H(\R)$ coincide with $\mP\mP_+(\R\P)$ and $\mP\mP_+(\R)$ respectively.
Relying on the fact that for any $A\neq \Z$, the group $\PSL(2,A)$ contains dense free subgroups, 
Monod proved in \cite{monod} that for any $A\neq \Z$, the group $H(A)$ is nonamenable. 
Therefore these groups give \emph{negative answer} to the Day-von Neumann problem.

\begin{rem}
The previous definition can be generalized, considering any subgroup $\Gamma\subset \PSL(2,\R)$. 
Elements in $G(\Gamma)$ are piecewise $\Gamma$ and the breakpoints are in $\cH_\Gamma$. 
For any non-discrete $\Gamma\subset\PSL(2,\R)$, the group $H(\Gamma)$ does not contain free subgroups and is nonamenable. Theorem~\ref{t:H(A)} can be extended to these groups as well.
\end{rem}

We shall now demonstrate  Theorem~\ref{t:H(A)}, namely that Monod's examples are not $C^1$-smoothable. For part \eqref{i:Monod1}, it is enough to prove the following: 
\begin{thm}\label{t:H(Z)}
Monod's group $H(\Z)$ is not $C^1$-smoothable.
\end{thm}
On the other hand, part \eqref{i:Monod2} relies on Theorem~\ref{t:main1Hyp}.

\begin{proof}[Proof of Theorem~\ref{t:H(A)}]
Let us first prove \eqref{i:Monod1}. Any subring $A\subset \R$ contains $\Z$, therefore $\PSL(2,\Z)$ is a subgroup of any $\PSL(2,A)$. Therefore we have inclusions $H(\Z)\subset H(A)\subset G(A)$. As $H(\Z)$ is not $C^1$-smoothable (Theorem \ref{t:H(Z)}), neither are $H(A)$ and $G(A)$.

\smallskip

Next, we demonstrate part \eqref{i:Monod2}. Let $\lambda>1$ be a \emr{Galois hyperbolic} number such that 
\[
f_\lambda:x\mapsto \lambda x
\]
belongs to $\PSL(2,A)$. As $\Z\subset A$, the translation $g:x\mapsto x+1$ belongs to $\PSL(2,A)$ as well.
Moreover, $f_\lambda$ being a hyperbolic element in $\PSL(2,A)$, we have that $\cH_A$ contains its fixed point $0$.
Therefore Monod's group $H(A)$ contains the piecewise defined element
\[
h_\lambda:x\mapsto\begin{cases}
x&\text{if }x\le 0,
\\
\lambda x&\text{if }x>0.
\end{cases}
\]
We have just shown that $G_\lambda=\langle f_\lambda,g,h_\lambda\rangle$ is a subgroup of $H(A)$.

Let $\rho:H(A)\to \Diff_+^1([0,1])$ be a representation. Theorem~\ref{t:main1Hyp} implies that 
$[g,h_\lambda gh_\lambda^{-1}]$ is in the kernel of $\rho$. Therefore $\rho$ cannot be injective, as desired.
\end{proof}


The dynamical ingredient we need for Theorem~\ref{t:H(Z)} is the following Sacksteder-like result, originally due to Deroin, Kleptsyn, Navas \cite{DKNacta} (see \cite[Prop.~3.2.10]{navas-book} and also \cite[\S~4.5]{centralizers} for a simplified proof).

\begin{prop}\label{p:linked_hyper}
Let $G=\langle f,g \rangle$ be a group acting by orientation-preserving $C^1$ diffeomorphisms on a compact one-dimensional manifold.
If $\{a,b\}$ and $\{c,d\}$ are linked pairs of successive fixed points for $f,g$, then $G$ contains an element with a hyperbolic fixed point in $(a,b)\cap(c,d)$.
\end{prop}

\begin{proof}[Proof of Theorem~\ref{t:H(Z)}]
Let us assume by way of contradiction that there is a homeomorphism $\phi:\R\to \R$ such that $G:=\phi H(\Z)  \phi^{-1}$ is
a group of $C^{1}$ diffeomorphisms of $\mathbb{R}$. 
First we observe that there are elements $f,g\in H(\Z)$ that have linked pairs of fixed points.
For example, consider the hyperbolic element $\gamma$ defined as the projective transformation 
\[ \gamma=\begin{bmatrix}
2 & -1 \\
-1 & 1 
\end{bmatrix},\]
whose fixed points $a,b$ satisfy that $a<-\frac{3}{2}<\frac{1}{2}<b$.
Now define 
\[
f(t)=\begin{cases}
t&\text{if }t\notin [a,b],\\
\gamma(t)&\text{if }t\in [a,b],
\end{cases}
\qquad g(t)=f(t-1)+1.\]
Note that the pairs $a,b$ and $a+1,b+1$ are linked.\footnote{We remark that in general a linked pair may not look like it does in this situation, for instance such maps may have components of support
lying outside $(a,b)$ and $(a+1,b+1)$ respectively.}

Now the elements
\[f_1=\phi f\phi^{-1},\qquad g_1=\phi g\phi^{-1}\]
in $G$ have fixed points $\phi(a),\phi(b)$ and $\phi(c),\phi(d)$ respectively. 
This forms a linked pair.
By Proposition~\ref{p:linked_hyper}, there is an element $g\in G$ with a fixed point such that the derivative of $g$ at $x$ is not equal to $1$.
Now let $g_1=\phi^{-1} g \phi$ be the corresponding element in $H(\Z)$.
Note that $g_1$ fixes $y=\phi^{-1}(x)$.

\smallskip

We claim that $y$ is a fixed point of a hyperbolic matrix in $\PSL(2,\Z)$.
If $y$ is a breakpoint of $g_1$, then this is true because the set of breakpoints of elements in $H(\Z)$ is exatly $\cH_{\Z}$.
We consider the case when $y$ is not a breakpoint of $g_1$, so there exists an element $\gamma_1\in\PSL(2,\Z)$ whose restriction to a neighbourhood $U$ of $y$ coincides with the restriction $g_1\vert_U$.

Observe that since $x$ is a hyperbolic fixed point for $g$, the corresponding point $y$ must be a topological attractor or repellor for $\gamma_1\in\PSL(2,\Z)$ that acts locally like $g_1$ around $y$,
and hence $g_1$ must be hyperbolic and $y$ is hence a hyperbolic fixed point for $\PSL(2,\Z)$.

\smallskip

Now consider an element $g_2\in H(\Z)$ which is the identity on $(-\infty, y)$ and agrees with $g_1$ on $[y,\infty)$.
Then $g_3=\phi g_2\phi^{-1}\in G$ has right derivative $\lambda\neq 1$ at $x$ and a left derivative that equals $1$ at $x$.
This contradicts the assumption that $g_3$ is $C^1$.
Hence our original assumption that $H(\Z)$ is $C^1$-smoothable must be false. 
\end{proof}

\subsection{The Lodha-Moore example}\label{sc:LM}

Lodha and Moore constructed a finitely presented subgroup $G_0$ of Monod's group.
This example provides the first torsion free finitely presentable example solving the Day-von Neumann problem.
The group $G_0$ is generated by $t\mapsto t+1$ together with the following
two homeomorphisms of $\R$:
\[
c(t)=
\begin{cases}
 t&\text{ if }t\leq 0,\\
 \frac{t}{1-t}&\text{ if }0\leq t\leq \frac{1}{2},\\
 3-\frac{1}{t}&\text{ if }\frac{1}{2}\leq t\leq 1,\\
 t+1&\text{ if }1\leq t,\\
\end{cases}
\qquad
d(t)=
\begin{cases}
 \frac{2t}{1+t}&\text{ if }0\leq t\leq 1,\\
t&\text{ if } t\notin [0,1].\\
\end{cases}
\]
The following was proved in \cite{LodhaMoore}:

\begin{thm}
The group $G_0$ is nonamenable and does not contain nonabelian free subgroups.
Moreover, it is finitely presentable with $3$ generators and $9$ relations. 
\end{thm}

In \cite{LodhaMoore} a combinatorial model for $G_0$ is constructed
by means of a faithful action of $G_0$ by homeomorphisms of the Cantor set $\{\0,\1\}^{\mathbb{N}}$.
This model was used to prove that $G_0$ is finitely presentable.
Here $\{\0,\1\}^{\N}$ is the Cantor set of infinite binary sequences, viewed as the boundary of the infinite rooted binary tree.
We denote by $\{\0,\1\}^{<\N}$ as the set of all finite binary sequences, which are addresses of nodes in the infinite rooted binary tree.

Consider the map $\Phi: \{\0,\1\}^\N \to \R\cup \{\infty\}$ given by:
\[\1\1^{a_0}\0^{a_1}1^{a_2}...\mapsto a_0+\frac{1}{a_1+\frac{1}{a_2+\ldots}},\qquad \0\0^{a_0}\1^{a_1}\0^{a_2}...\mapsto -\left (a_0+\frac{1}{a_1+\frac{1}{a_2+\ldots}}\right ).\]
This function is one-to-one except on sequences $\xi$ which are eventually constant.
On sequences which are eventually constant, the map is two-to-one:
$\Phi(s\0\1^{\infty}) = \Phi(s\1\0^{\infty})$ and $\Phi(\0^{\infty}) = \Phi(\1^{\infty}) = \infty$.

It was shown in \cite{LodhaMoore} that upon conjugating $a,b,c$ by $\Phi$ one obtains the following combinatorial model. 
We start with the following map $x:\{\0,\1\}^{\N}\to \{\0,\1\}^{\N}$ as:
\begin{align*}
x(\0\0\xi)&=\0\xi,\\
x(\0\1\xi)&=\1\0\xi,\\
x(\1\xi)&=\1\1\xi
\end{align*}
and also, recursively, the pair of mutually inverse maps $y,y^{-1}:\{\0,\1\}^{\N}\to \{\0,\1\}^{\N}$ as:

\begin{align*}
y(\0\0\xi)&=\0y(\xi),&y^{-1}(\0\xi)&=\0\0y^{-1}(\xi),\\
y(\0\1\xi)&=\1\0y^{-1}(\xi),&y^{-1}(\1\0\xi)&=\0\1y(\xi),\\
y(\1\xi)&=\1\1y(\xi),&y^{-1}(\1\1\xi)&=\1y^{-1}(\xi).
\end{align*}

From these functions, we define the functions $x_{s},y_s:\{\0,\1\}^{\N}\to \{\0,\1\}^{\N}$ for $s \in \{\0,\1\}^{<\N}$
which act as $x$ and $y$ localised to binary sequences which extend $s$:
\[
x_s(\xi)
 =
\begin{cases}
s x(\eta) & \textrm{ if } \xi = s \eta, \\
\xi & \textrm{otherwise},
\end{cases}
\qquad
y_s(\xi)
 =
\begin{cases}
s y(\eta) & \textrm{ if } \xi = s \eta, \\
\xi & \textrm{otherwise}.
\end{cases}
\]
If $s$ is the empty-string, it will be omitted as a subscript.
The group $G_0$ is generated by functions in the set
\[S=\left \{x_t,y_s\mid s,t\in \{\0,\1\}^{<\N},s\neq \0^k,s\neq \1^k, s\neq \emptyset\right \}\]
In fact, $G_0$ is generated by $x,x_{\1},y_{\1\0}$ which correspond respectively to conjugates of the functions $a,b,c$ defined above by $\Phi$.
(See \cite{LodhaMoore} for details.)

It is important to note that $G_0$ acts on the boundary of the infinite rooted binary tree, but not on the tree itself.

\smallskip
Recall from the introduction that we are denoting by $G_2$ the group generated by $f_2,g,h_2$, where $f_2$ is the scalar multiplication by $2$, $g$ is the translation by $1$, and $h_2$ is the element which agrees with $f_2$ to the right of zero and is the identity elsewhere.
We obtain the following obstruction to smoothability of $G_0$.

\begin{lem}\label{l:bbs_in_G0}
The three elements $y_{\1\0\0}^{-1}y_{\1\0\1},y_{\1\0\1}$ and $x_{\1\0}$ generate an isomorphic copy of $G_2$ in the Lodha-Moore group $G_0$.
\end{lem}

\begin{proof}
It was demonstrated in \cite{LodhaMoore} that the elements $x$ and $y_{\0}^{-1}y_\1$
are conjugate respectively to $t\mapsto t+1$ and $t\mapsto 2t$ by $\Phi$. Hence they generate an isomorphic copy of $\BS(1,2)$.
In particular,  $y_{\0}^{-1}y_{\1},y_{\1},x_{\1\0}$ generate an isomorphic copy of $G_2$.

It is easy to see that the groups $\left \langle y_{\1\0\0}^{-1}y_{\1\0\1},y_{\1\0\1},x_{\1\0}\right \rangle$ and $\left \langle y_{\0}^{-1}y_{\1},y_\1,x\right \rangle$ 
are isomorphic, since their respective actions on boundaries of the binary trees, $T_1$ rooted at the empty sequence and $T_2$ rooted at the sequence $\1\0$, are the same.
\end{proof}

More explicitly, one can verify that the elements $y_{\1\0\0}^{-1}y_{\1\0\1},y_{\1\0\1},x_{\1\0}$ correspond via $\Phi$ to the following piecewise projective transformations:
\begin{center}
\begin{minipage}{0.4\textwidth}
\[
x_{\1\0}\sim\begin{cases}
\begin{bmatrix}
1 & 0\\
-1 & 1
\end{bmatrix}&\text{on }\left[0,\frac{1}{3}\right ],\\
\\
\begin{bmatrix}
4 & -1\\
5 & -1
\end{bmatrix}&\text{on }\left[\frac{1}{3},\frac12\right ],\\
\\
\begin{bmatrix}
0 & 1\\
-1 & 2
\end{bmatrix}&\text{on }\left[\frac{1}{2},1\right ],\\
\\
id&\text{on }\R\setminus [0,1],
\end{cases}
\]
\end{minipage}
\begin{minipage}{0.4\textwidth}
\begin{align*}
y_{\1\0\1}\sim& \begin{cases}
\begin{bmatrix}
3 & -1\\
2 & 0
\end{bmatrix}&\text{on }\left [\frac{1}{2},1\right ],\\
\\
id&\text{on }\R\setminus \left [\frac12,1\right ],
\end{cases}
\\
\\
y_{\1\0\0}^{-1}y_{\1\0\1}\sim &
\begin{cases}
\begin{bmatrix}
1 & 0\\
-2 & 2
\end{bmatrix}&\text{on }\left [0,\frac{1}{2}\right ],\\
\\
\begin{bmatrix}
3 & -1\\
2 & 0
\end{bmatrix}&\text{on }\left [\frac12,1\right ],
\\
\\
id&\text{on }\R\setminus [0,1].
\end{cases}
\end{align*}
\end{minipage}
\begin{figure}[ht]
\[
\includegraphics[scale=0.6]{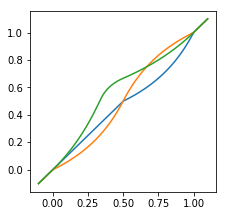}
\]
\caption{The generators $y_{\1\0\0}^{-1}y_{\1\0\1}$ (yellow), $y_{\1\0\1}^{-1}$ (blue), $x_{\1\0}$ (green) in a neighbourhood of $[0,1]$.}
\end{figure}
\end{center}

\begin{proof}[Proof of Theorem~\ref{t:LM}]
As a consequence of Lemma~\ref{l:bbs_in_G0}, the group $G_0$ contains a subgroup $H$ isomorphic to $G_2$.
Let $\rho:G_0\to\Diff_+^1([0,1])$ be a morphism. By a direct application of Theorem \ref{t:main1}, we obtain that the kernel of $\rho$ contains some nontrivial element of $H$. Thus $\rho$ is not injective. 

Now, it has been proven in \cite{commutators} that every proper quotient of $G_0$ is abelian, whence we get our result: as we have just shown that the kernel is not trivial, then the image must be abelian, as we wanted to prove.
\end{proof}

\subsection{Further examples}

An interesting family of nonamenable groups is obtained adding translations in top of $F$ (defined as in Definition \ref{ThurstonT}).
Mimicking Monod's argument, it is not difficult to prove the following:
\begin{prop}
For any $\alpha\in(0,1)$, the group of piecewise projective homeomorphisms generated by $F$ and the translation $t\mapsto t+\alpha$ is nonamenable.
\end{prop}

Observe that the groups $\langle F,t\mapsto t+\alpha\rangle$ appearing in the above statement are naturally of $C^1$~diffeomorphisms.

\begin{thm}\label{t:Firr}
For any irrational $\alpha\in (0,1)$, the action of the group of piecewise projective homeomorphisms $\langle F,t\mapsto t+\alpha\rangle$ on the compactified real line $[-\infty,+\infty]$ is not $C^2$-smoothable.
\end{thm}

\begin{proof}
We denote by $T_\alpha$ the translation by $\alpha$.
If $\alpha$ is irrational, then $T_1$ and $T_\alpha$ generate an abelian free group of rank $2$ of $C^2$ (even real analytic) diffeomorphisms of $\R$. The maps $f=T_{-1}$ and $g=T_{-|\alpha|}$ are contractions on $\R$. 
Consider any element $h\in \langle F,T_\alpha\rangle$ with a $C^2$ discontinuity point on $\R$. Then Theorem~\ref{t:nonC2} implies directly that the action of  $\langle F,T_\alpha\rangle$ on $[-\infty,+\infty]$ is not $C^2$-smoothable.
\end{proof}

For \emph{rational} translations $T_\alpha$, we can extend the previous argument and prove that even the action on the \emph{non-compactified} real line $(-\infty,+\infty)$ is not $C^2$-smoothable.
\begin{thm}\label{t:Frat}
For any rational $\alpha\in (0,1)$, the action of the group of piecewise projective homeomorphisms $\langle F,t\mapsto t+\alpha\rangle$ on $\R$ is not $C^2$-smoothable.
\end{thm}

\begin{proof}
We consider the conjugate of $c$ by $T_\alpha$:
\[
T_\alpha c T_{\alpha}^{-1}(t)=\begin{cases}
 t&\text{ if }t\leq \alpha,\\
 \frac{t-\alpha}{1-(t-\alpha)}+\alpha&\text{ if }\alpha\leq t\leq \frac{1}{2}+\alpha,\\
 3-\frac{1}{t-\alpha}+\alpha&\text{ if }\frac{1}{2}+\alpha\leq t\leq 1+\alpha,\\
 t+1&\text{ if }1+\alpha\leq t.\\
\end{cases}
\]
In restriction to the interval $\left[\alpha,\frac{1}{2}+\alpha\right]$, the element $T_\alpha c T_{\alpha}^{-1}$ coincides with the projective transformation
\[
\begin{bmatrix}
1-\alpha & \alpha^2 \\
-1 & 1+\alpha
\end{bmatrix},
\]
which is a parabolic element in $\PSL(2,\Z[\alpha])$ fixing $\alpha$. It is not in $\PSL(2,\Z)$.

Inside Thompson's $F$ we can find an element $f$ such that:
\begin{itemize}
\item $f$ fixes $\alpha$,
\item the restriction of $f$ to the interval $\left[ \alpha,\tfrac12+\alpha\right ]$ is $C^2$,
\item $f$ is a contraction of the interval $\left[ \alpha,\tfrac12+\alpha\right ]$, namely $f(t)<t$ for any $t$ in the right neighbourhood $\left( \alpha,\tfrac12+\alpha\right )$ of $\alpha$.
\end{itemize} 
Indeed, since $\alpha$ is rational, there exists a parabolic element in $\PSL(2,\Z)$ with $\alpha$ as fixed point, and we can take for $f$ any element of $F$ which coincides with this element (or its inverse) in restriction to $\left[ \alpha,\tfrac12+\alpha\right ]$.

Finally, consider an element $h\in F$ which has a $C^2$ discontinuity point $p$ on $\left[ \alpha,\tfrac12+\alpha\right ]$, with $h(p)\in \left[ \alpha,\tfrac12+\alpha\right ]$.

It is straightforward to verify that $f$, $g=\left (T_\alpha c T_{\alpha}^{-1}\right )^{-1}$ and $h$ satisfy the requirements of Theorem~\ref{t:nonC2} (when considering the interval $\left[ \alpha,\tfrac12+\alpha\right ]$ as the interval $[0,a]$ in the statement). Thus the theorem is proved.
\end{proof}

\section{$C^1$ actions of affine and piecewise affine groups}
\label{s:Aff}
\subsection{\emr{Baumslag-Solitar groups and affine groups}}
Let $n>1$ be an integer. 
The classical \emph{Baumslag-Solitar} groups $\BS(1,n)$ are defined by the presentations
\[
\BS(1,n)=\left \langle a,b\mid aba^{-1}=b^n\right \rangle.
\]
They are naturally realized as subgroups of the affine group $\mathsf{Aff}_+(\R)\subset \PSL(2,\R)$, generated by the homothety $a(x)=n x$, and the translation $b(x)=x+1$.

Similarly, for any rational $\lambda=p/q>1$ there is a morphism from the Bausmlag-Solitar group
\[
\BS(q,p)=\left \langle a,b\mid ab^qa^{-1}=b^p\right \rangle
\] 
to the subgroup $A_\lambda$ of $\mathsf{Aff}_+(\R)$ generated by $a(x)=\lambda x$ and $b(x)=x+1$. However, when $p/q>1$ is not an integer, this morphism is not an isomorphism. 
\emr{For general $\lambda>1$, we define $A_\lambda$ to be the subgroup of $\mathsf{Aff}_+(\R)$ generated by $a(x)=\lambda x$ and $b(x)=x+1$.
Observe that the conjugate $aba^{-1}$ equals the translation $x\mapsto x+\lambda$, and hence $b$ and $ab a^{-1}$ commute in $A_\lambda$ (this is not true for nonsolvable Baumslag-Solitar groups $\BS(q,p)$, $p/q>1$ not integer). The group $A_\lambda$ is \emph{abelian by cyclic}, abstractly isomorphic to the semi-direct product $\Z[\lambda,\lambda^{- 1}]\rtimes \Z$ (where $\Z$ acts on $\Z[\lambda,\lambda^{- 1}]$ by multiplication by $\lambda$). 
More precisely, for $\lambda$ transcendental, $A_\lambda$ is isomorphic to the wreath product $\Z\wr \Z\cong \Z[t,t^{-1}]\rtimes \Z$, whereas if algebraic the following properties hold.
\begin{lem}
	\label{lem:abel_aff}
	Let $\lambda>1$ be an algebraic real number of degree $d$, and let $p_\lambda(t)=\frac{\a_0}{\a_d}+\frac{\a_1}{\a_d}t+\ldots +\frac{\a_{d-1}}{\a_d}t^{d-1}+t^d$, $\a_j\in \Z$, denote the associated minimal polynomial.
	\begin{enumerate}[(1)]
	\item \label{item:Qrank} The group $H=\Z[\lambda,\lambda^{-1}]$ has \emph{$\Q$-rank} equal to $d$: it is an additive subgroup of $\Q(\lambda)\cong \Q^d$, and does not embed in a $\Q$-vector space of lower dimension.
	\item \label{item:matrix} With respect to the basis $\{1,\lambda,\ldots,\lambda^{d-1}\}$, the homothety $a$ acts on $H$ as multiplication by the companion matrix $C_\lambda$, so that one has $A_\lambda\cong H\rtimes_{C_\lambda} \Z$.
	\item  The group $A_\lambda$ is a quotient of the finitely presented group
	\begin{equation}\label{PresentationAff}
	\widehat A_\lambda=\left\langle \widehat a, b_0,\ldots,b_{d-1} \left\vert
	\begin{array}{lr}
	b_ib_j=b_jb_i&  \text{for }i,j=0,\ldots,d-1,\\
	\widehat ab_j\widehat a^{-1}=b_{j+1} & \text{for }j=0,\ldots, d-2,\\ 
	\widehat ab_{d-1}^{\a_d}\widehat a^{-1}=b_0^{-\a_0}\cdots b_{d-1}^{-\a_{d-1}} & \\
	\end{array}
	\right.
	\right\rangle,
	\end{equation}
	where generators $b_j$ are mapped to the translations $x\mapsto x+\lambda^j$ in the standard affine action, and $\widehat a$ to the homothety of factor $\lambda$.
	\item\label{item:abelian} The abelianization of $\widehat A_\lambda$ is the abelian group $\Z\times \Z/(\a_d\, p_\lambda(1)\Z)$, where the factor $\Z$ is generated by the image of $\widehat a$, and the finite factor $\Z/(\a_d\, p_\lambda(1)\Z)$ is generated by the image of any $b_j$.
	In particular, any torsion-free abelian quotient of $A_\lambda$ is either trivial or infinite cyclic.
	\end{enumerate}
\end{lem}
The proof being elementary, we rather omit it. For the last statement, observe that $p_\lambda$ is the minimal polynomial of $\lambda$ (which is a real algebraic number $\neq 1$), hence $1$ cannot be a root and therefore $\a_d\, p_\lambda(1)$ is always a nonzero integer.}

\subsection{$C^1$ actions of \emr{abelian-by-cyclic groups}}

\emr{In \cite{FarbFranksI} Farb and Franks, relying on Kopell's lemma, show that every $C^2$ action of $\BS(q,p)$ on one-dimensional manifolds quotients through an action of its image $\Z[p/q,q/p]\rtimes \Z$ in $\mathsf{Aff}_+(\R)$. To the best of our knowledge, nothing appears in the literature about actions in lower regularity.}

\emr{The reason why actions of (solvable) Baumslag-Solitar groups are widely studied is because of the simple presentation, given by just one relation, $ab^ma^{-1}=b^n$, which has a dynamical meaning: $a$ conjugates a power of $b$ to another power. One of the first relevant works in this subject is the aforementioned \cite{FarbFranksI}, where the authors study general actions of $\BS(q,p)$ on one-manifolds. This was pursued by Burslem-Wilkinson \cite{BW}, where they study sufficiently regular actions of $\BS(1,n)$ on the circle. Later improvements are due to Guelman-Liousse \cite{GL}, and finally to Bonatti-Monteverde-Navas-Rivas \cite{hyperbolic}. For actions on higher-dimensional manifolds, McCarthy \cite{McCarthy} proved that $C^1$ perturbations of the trivial action of torsion-free, finitely presented, abelian-by-cyclic groups are not faithful. Another example of rigidity result was obtained by Asaoka \cite{A1,A2} for standard actions of the same class of groups on spheres and tori, and also by Wilkinson-Xue \cite{WX} for actions on tori. Finally, planar actions of $\BS(1,n)$ have been investigated by several authors \cite{GL2,ARX,AGRX}.}

In relation with our work, Bonatti-Monteverde-Navas-Rivas study the $C^1$ actions on the interval of abelian-by-cyclic groups like $A_\lambda$. 
The following result appears in \cite[\S~4.3]{hyperbolic} (even if not explicitly stated for general $\lambda>1$, the arguments in \cite[\S~4.3]{hyperbolic} only use the condition $\lambda\ge 2$, which is always guaranteed, up to taking an integer power of $a$):
\begin{prop}\label{p:affine}
Fix an arbitrary $\lambda>1$ and let $\phi:\R\to\R$ be a homeomorphism such that $\phi A_\lambda\phi^{-1}$ is in $\Diff_+^1([0,1])$. 
Then $\phi a\phi^{-1}$ has derivative equal to $\lambda$ at its interior fixed point $\phi(0)$.
\end{prop}

\emr{For \emr{Galois hyperbolic} $\lambda>1$, for instance as $\lambda>1$ is rational, we obtain from \cite{hyperbolic} a much stronger statement:
\begin{thm}\label{t:BMNR}
Let $\lambda>1$ be a Galois hyperbolic number and consider a $C^1$ action of the abelian-by-cyclic group $A_\lambda$ on the closed interval, without global fixed points in its interior. 
	Then every nonabelian action of $A_\lambda$ is topologically conjugate to the standard affine action.
\end{thm}
\begin{proof}
	Properties \eqref{item:Qrank} and \eqref{item:matrix} in Lemma~\ref{lem:abel_aff} guarantee that the hypotheses of \cite[Thm.~1.3]{hyperbolic} are satisfied for the group $A_\lambda$, provided $\lambda>1$ is Galois hyperbolic. This gives that any $C^1$ action of $A_\lambda$ on the closed interval, without global fixed points in its interior, is topological conjugate to a representation of $A_\lambda$ into the affine group $\mathsf{Aff}_+(\R)$. Representations $\psi:A_\lambda \to \mathsf{Aff}_+(\R)$ are classified by \cite[Prop.~2.1]{hyperbolic}: when the image of $\psi$ is nonabelian, (1) the generator $a$ of $A_\lambda$ (the homothety of factor $\lambda$) is mapped to itself and (2) the generator $b$ of $A_\lambda$ (the translation) is mapped to some translation. Therefore $\psi$ is conjugate to the standard affine action.
\end{proof}
\begin{rem}\label{couterex}
Observe that the statement above cannot be true as $\lambda>1$ is transcendental, because $A_\lambda\cong \Z\wr\Z$ has many distinct actions on the interval. Moreover, in \cite[\S~5]{hyperbolic}, even as $\lambda>1$ is algebraic but not Galois hyperbolic, there are further $C^1$ actions of $A_\lambda$.
\end{rem}
}

\subsection{The groups $G_\lambda$}

Inspired by the definition of Monod's groups, we consider an analogous construction starting from these affine groups.
Here we repeat the definition already given in the introduction:

\begin{dfn}\label{BrokeBS}
For any $\lambda>1$, we define $G_\lambda$ to be the subgroup of $\mP\mP_+(\R)$ generated by the elements
\[
a(x)=\lambda x,\quad
a_+(x)=\begin{cases} x & \text{if }x\le 0\\
\lambda x & \text{if }x>0\end{cases},\quad b(x)=x+1.
\]
We also set $a_-=aa_+^{-1}$, which agrees with $a$ to the left of $0$ and is the identity elsewhere.
\end{dfn}

\begin{rem}
In the introduction, we were denoting $a,b,a_+$ by $f_\lambda,g,h_\lambda$ respectively.
\end{rem}

\emr{\begin{lem}\label{lem:abel_G} Let $\lambda>1$ be an algebraic number.
	The image of the generator $b\in G_\lambda$ is trivial in any torsion-free abelian quotient of $G_\lambda$. Indeed, any such quotient is either trivial, or infinite cyclic, or isomorphic to $\Z^2$, generated by the images of $a_\pm$.
\end{lem}
\begin{proof}
	By \eqref{item:abelian} in Lemma~\ref{lem:abel_aff}, every image of the generator $b$ in an abelian group must be of finite order.
\end{proof}}

\begin{rem}\label{r:F}
The algebraic structure of $G_\lambda$ is highly complicated. 
For instance, in the case $\lambda=2$, inside the group $G_2$, the elements $b$ and $[a_+,b]$ are the generators of Thompson's $F$, in its natural piecewise linear action on $\R$.
\emr{In fact, every group $G_{\lambda}$ contains a copy of $F$.
To see this, let 
\[f_1=b^{-1} a_+ b,\qquad f_2=b a_- b^{-1}.\]
The (open) support of $f_1$ is the half-line $J_1=(-1,+\infty)$ whereas the support of $f_2$ is the half-line $J_2=(-\infty,1)$. These supports form a \emph{chain} $(J_1,J_2)$ in the sense of \cite{chains}. Then by \cite[Thm.~3.1]{chains}, there exists $n\in \mathbf{N}$ such that $\langle f_1^n,f_2^n\rangle$ is isomorphic to Thompson's group $F$.}
\end{rem}

\begin{ex}
There are two canonical standard affine actions of the group $G_\lambda$ on the real line that factor through the affine group $A_\lambda$.
First, as every element in $G_\lambda$ fixes $\pm\infty$, we can consider the germs of elements of $G_\lambda$ at these two points. 
This gives us two surjective homomorphisms
\[
\rho_\pm:G_\lambda\to A_\lambda.
\]
It is clear from the definition of $G_\lambda$ that we have
\[
\rho_\pm(a_\mp)=id,\quad \rho_\pm(a_\pm)=\rho(a)
\]
for these two morphisms. More generally, every element of $G_\lambda$ that is the identity outside a compact interval belongs to the kernels of both morphisms $\rho_\pm$. This is the case for the commutator $[b,a_+ba_+^{-1}]$ that appears in the statement of Theorem~\ref{t:main1Hyp}.

\emr{On the other hand, as the abelianization of $G_\lambda$ is not trivial, there are plenty of abelian actions of $G_\lambda$ on the real line. Recall thats any group of orientation preserving homeomorphisms of the real line is torsion free, therefore as $\lambda>1$ is algebraic, Lemma~\ref{lem:abel_G} implies that the generator $b$ must be in the kernel of any abelian action. In particular, as $\lambda>1$ is algebraic, the commutator $[b,a_+ba_+^{-1}]$ always acts trivially.}
\end{ex}

\subsection{$C^1$ actions of  $G_\lambda$}

The following result is essentially the one contained in Theorem~\ref{t:main2}:

\begin{thm}\label{t:affine}
For any $\lambda>1$, the natural action of $G_\lambda$ on the compactified real line $[-\infty,+\infty]$ is not $C^1$-smoothable.
\end{thm}

\begin{proof}
We argue by contradiction. After Proposition~\ref{p:affine}, if there existed a homeomorphism $\phi:\R\to\R$ such that $\phi A_\lambda\phi^{-1}$ was in $\Diff_+^1([-\infty,+\infty])$, 
then $\phi a\phi^{-1}$ would have derivative equal to $\lambda$ at $p=\phi(0)$ and $\phi a_+ \phi^{-1}$ would not be $C^1$ at $p$. Hence the group is not $C^1$-smoothable.
\end{proof}

\begin{rem}
In the previous statement, it is fundamental to consider the action of $G_\lambda$ on the compactified line. Indeed, the statement is no longer true if one simply considers the action on $\R$ (see \cite[Remark 4.14]{hyperbolic}).
\end{rem}

Our second result, more precise than the statement in Theorem~\ref{t:main1Hyp}, says that every $C^1$ action of $G_\lambda$ on the interval, for \emr{Galois hyperbolic} $\lambda>1$, is always described by combining the examples above.

\begin{thm}\label{t:key}
\emr{Let $\lambda>1$ be Galois hyperbolic and} let $\rho:G_\lambda\to \Diff^1_+([0,1])$ be a nontrivial homomorphism. Then there exists finitely many pairwise disjoint subinterval $I_1,\ldots,I_n\subset [0,1]$ such that
\begin{enumerate}[(1)]
\item for any $i=1,\ldots,n$, the image $\rho(G_\lambda)$ preserves the interval $I_i$,
\item for any $i=1,\ldots,n$, the restriction of $\rho(G_\lambda)$ to $I_i$ is topologically conjugate to one of the two canonical actions on $\R$.
\item the restriction of $\rho(G_\lambda)$ to the complement $[0,1]\setminus \bigcup_{i=1}^n I_i$ is abelian.
\end{enumerate}
In particular the group $G_\lambda$ admits no faithful $C^1$ action on the closed interval.
\end{thm}

\begin{rem}
Relying on \cite[Thm.~1.10]{hyperbolic} (\textit{cf.} also \cite{GL}), we could provide a similar statement for $C^1$ actions of $G_\lambda$ on the circle $\T$.
Indeed, every nonabelian action of $A_\lambda$ has a global \emr{finite orbit} so, \emr{up to passing to a finite-index subgroup}, every nonabelian action of $G_\lambda$ reduces to an action on the interval.
\end{rem}

\begin{rem}\label{r:C1+}
The proof of Theorem~\ref{t:key} would be much simpler for representations $\rho:G_\lambda\to \Diff^{1+\alpha}_+([0,1])$ of the group $G_\lambda$ into the group of $C^1$ diffeomorphisms with $\alpha$-H\"older continuous derivative. Indeed, it is classical that any $C^1$ element commuting with a $C^{1+\alpha}$ hyperbolic contraction  of an interval 
lies in a one parameter flow (\textit{cf.}~Theorem~\ref{t:Szekeres}: when the fixed point of the contraction is hyperbolic, Szekeres theorem requires only $C^{1+\alpha}$ regularity).

Let us sketch the proof under the assumption of $C^{1+\alpha}$ regularity. Assume that the image $\rho(G_\lambda)$ is nonabelian. Then the image $\rho(A_\lambda)$ is also nonabelian (\textit{cf.}~Lemma \ref{l:equiv1}). From Theorem~\ref{t:BMNR} and Proposition~\ref{p:affine} we deduce that the element $\rho(a)$ behaves as the corresponding scalar multiplication in restriction to some interval $I\subset [0,1]$ and has a hyperbolic fixed point $s\in I$. 
As the elements $\rho(a_\pm)$ commute with $\rho(a)$, we deduce from Szekeres theorem that in restriction to the interval $I$, also these elements behave like scalar multiplications (as the one parameter flow containing a scalar multiplication is exactly the one parameter flow of all scalar multiplications). This implies that the group $\rho(G_\lambda)$ acts like an affine group in restriction to the interval $I$.
\end{rem}

The proof of Theorem~\ref{t:key} will occupy the rest of the section.

\subsection{Elementary ingredients}

When working with $C^1$ actions on the interval, hyperbolic fixed points do not often give rigidity (one usually needs $C^{1+\alpha}$ regularity, \textit{cf.}~Remark \ref{r:C1+}). Indeed there are only a few dynamical tools that work in $C^1$ regularity. For this reason our proof relies mainly on very elementary arguments.
A first tool is the following:

\begin{lem}\label{FinitelyMany}
Let $\alpha\in\Diff_+^1([0,1])$ be a diffeomorphism. For any $\delta>1$, there are only finitely many points $s\in [0,1]$
that are fixed by $\alpha$ and such that $\alpha'(s)>\delta$.
\end{lem}

\begin{proof}
Suppose $\alpha$ has infinitely many fixed points $\{s_n\mid n\in \N\}$ in $[0,1]$, such that $\alpha'(s_n)>\delta$ for any $n\in\N$.
Let $s_*\in [0,1]$ be an accumulation point of the sequence $\{s_n\}$. By continuity of the derivative, we must have
$\alpha'(s_*)\ge\delta$. On the other hand, let $\{s_{n_k}\}$ be a subsequence converging to $s_*$; by the very definition of the derivative
we must have $\alpha'(s_*)=1$. This is a contradiction.
\end{proof}

Then we state and prove a second crucial elementary fact.

\begin{lem}\label{CommutingHyp}
Let $\alpha,\beta \in \Diff_+^1([0,1])$ be two commuting $C^1$ diffeomorphisms.
Let $s\in [0,1]$ be a hyperbolic fixed point of $\alpha$.
Then $\beta$ fixes $s$.
\end{lem}

\begin{proof}
Let us assume by way of contradiction that $\beta$ does not fix $s$.
For each $n\in \Z$ we have 
\[\alpha(\beta^n(s))=\beta^n(\alpha(s))=\beta^n(s)\]
This means that $\alpha$ fixes each point in the set $S=\{\beta^n(s)\mid n\in \Z\}$.

\begin{claim}
Each $t\in S$ is a hyperbolic fixed point of $\alpha$ and $\alpha'(t)=\alpha'(s)$ for all $t\in S$.
\end{claim}

\begin{proof}[Proof of Claim]
Let $\lambda_n$ be the formal word $\beta^{-n} \alpha \beta^{n}$. 
Using the chain rule, we find
\[\lambda_n'(s)=\alpha'(\beta^n(s)).\]
However, since $\alpha$ and $\beta$ commute, indeed $\lambda_n=\alpha$ and hence $\lambda_n'(s)=\alpha'(s)$.
It follows that $\alpha'(s)=\alpha'(\beta^n(s))$ for each $n\in \mathbb{N}$.
\end{proof}

Since the set $S$ is infinite, the claim is in contradiction with Lemma~\ref{FinitelyMany}.
\end{proof}

\subsection{A particular case: no global fixed points for $A_\lambda$}

Before dealing with a general statement as in Theorem~\ref{t:key}, we study actions on the interval without global fixed points. 
For the statement, recall that we denote by $A_\lambda\subset G_\lambda$ the subgroup generated by $a$ and $b$.

\begin{prop}\label{p:noglobal}
Let $\lambda>1$ be \emr{Galois hyperbolic} and  $\rho:G_\lambda\to \Diff_+^1([0,1])$ be a morphism satisfying the following:
\begin{enumerate}[(1)]
\item the image $\rho(A_\lambda)$ is nonabelian,
\item the action of $\rho(A_\lambda)$ has no global fixed point in $(0,1)$.
\end{enumerate}
Then $\rho(G_\lambda)$ is topological conjugate to one of the two canonical representations $\rho_\pm:G_\lambda\to A_\lambda$.
\end{prop}

In the following, we let $s_0$ denote the  
hyperbolic fixed point of $\rho(a)$ in $(0,1)$, ensured by Theorem~\ref{t:BMNR} and Proposition~\ref{p:affine}
above.
For simplicity of notation, we also write
\[
\rho(a)=f,\quad \rho(b)=g, \quad \rho(a_-)=h,\quad \rho(a_+)=k.
\]

\begin{lem}\label{l:bbs-derivative}
With the notation as above, the elements $h$ and $k$ fix the point $s_0$ and we have
\[
h'(s_0)\cdot k'(s_0)=\lambda.
\]
\end{lem}
\begin{proof}
By Lemma~\ref{CommutingHyp}, the two elements $h,k$ fix the point $s_0$, as they commute with $f$.
Remarking that $hk=f$, applying the chain rule we deduce that 
\begin{align*}
\lambda&=f'(s_0)\\
&=h'(k(s_0))\cdot k'(s_0)\\
&=h'(s_0)\cdot k'(s_0),
\end{align*}
as wanted.
\end{proof}

The previous lemma implies that $s_0$ is a hyperbolic fixed point for at least one among~$h$
 and~$k$. Without loss of generality, we assume $h'(s_0)>1$. The following lemma says
that $h$ behaves like a hyperbolic element on the whole interval $[0,1]$:
\begin{lem}\label{l:bbs2}
With the notation as above, suppose $h'(s_0)>1$. Then $s_0$ is the only point of $(0,1)$ which is fixed by $h$. 
\end{lem}
\begin{proof}
If $h$ had a fixed point $s$ different from $s_0$, since $h$ and $f$ commute, the images $f^{-n}(s)$ would form a sequence
of fixed points for $h$ that converge
to $s_0$. This would imply that the derivative of $h$ at $s_0$ should be equal to $1$. Contradiction.
\end{proof}

Given  \emr{$r\in \Z[\lambda,\lambda^{-1}]$}, we denote by $g_r$ the image by $\rho$ of the translation by $r$ when thinking of $A_\lambda$ 
as an affine group. We have $g_{-r}=g_r^{-1}$. By Theorem~\ref{t:BMNR}   these elements 
are actually topologically conjugate to the corresponding translations.
 
\begin{lem}\label{l:bbs3}
Take \emr{a positive $r\in \Z[\lambda,\lambda^{-1}]$}. The conjugate $g_{r}kg_r^{-1}$ commutes with $h$. 
\end{lem}
\begin{proof}
This is actually a statement about relations of the group $G_\lambda$: 
we prove the relation looking at the standard action of $G_\lambda$ on the real line. The support of $a_+$ is $[0,+\infty)$,
therefore the support of the conjugate of $a_+$ by the translation by $r$ is $[r,+\infty)$, which is disjoint from
$(-\infty,0]$, which is the support of $a_-$.
\end{proof}

\begin{lem}\label{l:bbs4}
With the notation as above, the restriction of $k$ to $[0,s_0]$ is the identity.
\end{lem}
\begin{proof}
The element $k_r=g_{r}kg_r^{-1}$ commutes with $h$ by Lemma~\ref{l:bbs3}.
\emr{As $s_0$ is a hyperbolic fixed point for $h$, Lemma~\ref{CommutingHyp} implies that  $k_r$ fixes $s_0$ for each $r>0$.}
In particular, it follows that $k$ fixes $g_{-r}(s_0)$ for every positive $r\in \Z[\lambda,\lambda^{-1}]$. By density of \emr{$\Z[\lambda,\lambda^{-1}]$} in $(0,+\infty)$, we obtain the statement.
\end{proof}

The end of the proof is inspired by \cite{centralizers}: in a centraliser of a hyperbolic element, like $h$, there cannot be elements with hyperbolic fixed points (different from the fixed points of the hyperbolic element), and therefore by 
Proposition~\ref{p:linked_hyper}, there cannot be linked pairs of successive fixed points.

\begin{lem}\label{l:bbs5}
Suppose that $k$ is not the identity and let $s, t\in [s_0,1]$ be a pair of successive fixed points of $k$. 
There exists a \emr{positive $r\in \Z[\lambda,\lambda^{-1}]$} such that the pair $g_r(s),g_r(t)$ together with $s,t$ defines a linked pair of fixed points 
for  $g_rkg_r^{-1}$ and $k$.
\end{lem}

\begin{proof}
For any \emr{positive $r\in \Z[\lambda,\lambda^{-1}]$}, the points $g_r(s),g_r(t)$  define a pair of successive fixed points for the conjugate $g_rkg_r^{-1}$.
An element $g_r$, $r>0$, moves every point in $(s_0,1)$ to the right, and using the fact that $g_r$ is topologically conjugate to the translation
by $r$, we can choose $r>0$ sufficiently small such that
\[
s<g_r(s)<t\le g_r(t)
\]
(with equality $t=g_r(t)$ if and only if $t=1$).
\end{proof}

Suppose that $k$ is not the identity. Then, from Lemma~\ref{l:bbs5} and Proposition~\ref{p:linked_hyper},
we realize that the subgroup $\langle g_rkg_r^{-1},k\rangle$ contains an element $\gamma$ 
with a hyperbolic fixed point $p$ in $(s_0,1)$. 
Since $g_rkg_r^{-1},k$ commute with $h$, it follows that $\gamma$ commutes with $h$. 
So by Lemma \ref{CommutingHyp} $h$ must fix the point $p$.
This contradicts the conclusion of Lemma~\ref{l:bbs2}.

Therefore we must have that $k=\rho(a_+)$ is the identity, and so $\rho(a_-)=\rho(a)$. Thus the representation $\rho:G_\lambda\to \Diff_+^1([0,1])$
is topologically conjugate to the canonical representation $\rho_-:G_\lambda\to A_\lambda$.
This finishes the proof of Proposition \ref{p:noglobal}.

\subsection{Equivalent properties}

Now we consider almost the same statement as in Proposition \ref{p:noglobal}, but we only make assumptions on the global dynamics of $G_\lambda$, rather than on the one of $A_\lambda$.

\begin{prop}\label{p:noglobal2}
Let $\lambda>1$ be \emr{Galois hyperbolic} and  $\rho:G_\lambda\to \Diff_+^1([0,1])$ be a morphism satisfying the following:
\begin{enumerate}[(1)]
\item the image $\rho(G_\lambda)$ is nonabelian,
\item the action of $\rho(G_\lambda)$ has no global fixed point in $(0,1)$.
\end{enumerate}
Then $\rho(G_\lambda)$ is topological conjugate to one of the two canonical representations $\rho_\pm:G_\lambda\to A_\lambda$.
\end{prop}

The proof follows directly from the following two lemmas \emr{and from Proposition~\ref{p:noglobal}}.

\begin{lem}\label{l:equiv1}
\emr{Let $\lambda>1$ be algebraic} and $\rho:G_\lambda\to \Diff_+^1([0,1])$ be a morphism. Then the following properties are equivalent:
\begin{enumerate}[(1)]
\item the image $\rho(G_\lambda)$ is nonabelian,
\item the image $\rho(A_\lambda)$ is nonabelian.
\end{enumerate}
\end{lem}

\begin{proof}
Clearly (2) implies (1). On the other hand, if the image $\rho(A_\lambda)$ is abelian, \emr{\eqref{item:abelian} in Lemma~\ref{lem:abel_aff} implies that the translation $b$ is in the kernel of $\rho$, and as in Lemma~\ref{lem:abel_G}}, $\rho(G_\lambda)$ itself is abelian.
\end{proof}

\begin{lem}\label{l:equiv2}
Let $\lambda>1$ be \emr{Galois hyperbolic} and $\rho:G_\lambda\to \Diff_+^1([0,1])$ be a morphism with nonabelian image. Then the following properties are equivalent:
\begin{enumerate}[(1)]
\item the action of $\rho(G_\lambda)$ has no global fixed point in $(0,1)$,
\item the action of $\rho(A_\lambda)$ has no global fixed point in $(0,1)$.
\end{enumerate}
\end{lem}
\begin{proof}
Again, (2) easily implies (1). 
Assume (1). Since the image $\rho(G_\lambda)$ is nonabelian, by Lemma \ref{l:equiv1} also the image $\rho(A_\lambda)$ is nonabelian. Using Theorem~\ref{t:BMNR}, we find at least one interval $I=[x,y]\subset [0,1]$ such that
\begin{enumerate}[(1)]
\item $I$ is preserved by $\rho(A_\lambda)$,
\item $\rho(A_\lambda)$ has no global fixed point in the interior of $I$,
\item the restriction $\rho(A_\lambda)\vert_I$ is nonabelian and therefore it is topologically conjugate to the standard affine action on $\R$.
\end{enumerate}
Moreover, by Proposition \ref{p:affine}, there exists a unique point $s_0\in (x,y)$ in the interior of $I$ which is a hyperbolic fixed point for $\rho(a)$,
with derivative $\rho(a)'(s_0)=\lambda$.

Proceeding as in Lemma~\ref{l:bbs-derivative}, the elements $\rho(a_\pm)$ must fix the point $s_0$ and we can suppose that
$s_0$ is a hyperbolic fixed point for $\rho(a_-)$, with derivative $\rho(a_-)'(s_0)>1$. Let $s_-$ be the first fixed point of $\rho(a_-)$ which lies to the left of $s_0$.

If $s_-\in (x,s_0)$, then $\{\rho(a)^{-n}(s_-)\mid n\in \N\}$ is a sequence of fixed points for $\rho(a_-)$ that converges to $s_0$ as $n\to\infty$. But this is not possible
because the derivative of $\rho(a_-)$ at $s_0$ is not $1$ (\textit{cf.}~Lemma \ref{CommutingHyp}).

Similarly, if $s_-\in [0,x)$, then $\{\rho(a_-)^{-n}(x)\mid n\in \N\}$ is a sequence of fixed points for $\rho(a)$ that converges to $s_0$ as $n\to\infty$. Again, this is not
possible.

Thus $s_-=x$ and so $x$ is a global fixed point for $\rho(G_\lambda)$. As we are assuming (1), this implies $x=0$.
Similarly, denoting by $s_+$ the first fixed point of $\rho(a_-)$ which lies to the right of $s_0$, we obtain that $s_+=y$ and so $y=1$. This is what we wanted to prove.
\end{proof}

\begin{proof}[Proof of Proposition~\ref{p:noglobal2}]
The statement follows directly from Lemmas~\ref{l:equiv1} and \ref{l:equiv2}, \emr{and from Proposition~\ref{p:noglobal}}.
\end{proof}

\subsection{General case}

We proceed now to the proof of Theorem~\ref{t:key}.

\begin{proof}[Proof of Theorem \ref{t:key}]
Let $\rho:G_\lambda\to \Diff^1_+([0,1])$ be a homomorphism. 
If the image $\rho(G_\lambda)$ is abelian, there is nothing to prove. Hence we can assume that $\rho(G_\lambda)$ is nonabelian
and by Lemma~\ref{l:equiv1} this is equivalent to saying that $\rho(A_\lambda)$ is nonabelian, where $A_\lambda$ is denoting the subgroup
generated by $a$ and $b$.

If $\rho(A_\lambda)$ is nonabelian, then after Theorem~\ref{t:BMNR}, there exists at least an interval $I\subset [0,1]$ which is preserved
by $\rho(A_\lambda)$ and such that $\rho(A_\lambda)$ is topologically conjugate to the standard affine action.

\begin{claim}
There are only finitely many pairwise disjoint intervals $I_1,\ldots, I_n$ that are preserved by $\rho(A_\lambda)$ and such that for any $i=1,\ldots,n$
the restrictions $\rho(A_\lambda)\vert_{I_i}$ is nonabelian.
\end{claim} 

\begin{proof}[Proof of Claim]
Let $I$ be an interval preserved by $\rho(A_\lambda)$ and such that $\rho(A_\lambda)\vert_I$ is nonabelian. By Theorem~\ref{t:BMNR}, the action is topologically
conjugate to the standard action of $A_\lambda$ and by Proposition~\ref{p:affine} there exists a point $s\in I$ which is fixed by $\rho(a)$ and such that
$\rho(a)'(s_n)=\lambda>1$.
Then Lemma~\ref{FinitelyMany} implies that there can only be finitely many such intervals, whence the first statement.
\end{proof}

\begin{claim}
Let $I_1,\ldots,I_n$ be the intervals provided by the previous claim. Then $\rho(G_\lambda)$ preserves $I_i$ for any $i=1,\ldots,n$.
\end{claim}

\begin{proof}[Proof of Claim]
Let $I$ be an interval as above. Let $J\subset I$ be a interval which is preserved by $\rho(G_\lambda)$ and such that
$\rho(G_\lambda)$ has no global fixed point in its interior. By Proposition~\ref{l:equiv2}, we must have the equality $I=J$.
\end{proof}

After Proposition~\ref{p:noglobal2}, we deduce that the restriction of the action of $G_\lambda$ to any
of the intervals $I_1,\ldots,I_n$ is topologically conjugate to one of the two canonical affine actions that filters through a quotient $\rho_\pm:G_\lambda\to A_\lambda$.
This is what we wanted to prove.
\end{proof}

\section{$C^2$ actions with locally non-discrete stabilisers}
\label{s:C2}

\subsection{Szekeres vector field}

The method that we present in this section is inspired by \cite[Prop.~4.17]{hyperbolic} and relies on the following important result in one-dimensional dynamics, due to Szekeres \cite{szekeres}. Here we state it as in \cite[\S~4.1.3]{navas-book}:

\begin{thm}[Szekeres]\label{t:Szekeres}
Let $f$ be a $C^2$ diffeomorphism of the half-open interval $[0,1)$ with no fixed point in $(0,1)$. Then there exists a unique $C^1$ vector field $\mathcal X$ on $[0,1)$ with no singularities on $(0,1)$ such that
\begin{enumerate}[1)]
\item $f$ is the time-one map of the flow $\{\phi_{\mathcal X}^s\}$ generated  by $\mathcal X$,
\item the flow $\{\phi_\mathcal X^s\}$ coincides with the $C^1$ centraliser of $f$ in $\Diff_+^1([0,1))$.
\end{enumerate}
\end{thm}

\subsection{An obstruction to $C^2$ smoothability}

The criterion we provide holds in a framework that is far more general that the one of piecewise-projective dynamics. First we need a statement of differentiable rigidity for the conjugacy of some particular actions.

\begin{prop}\label{p:nonC2}
Take $a\in (0,1)$ and assume that two homeomorphisms $f,g\in \Homeo_+([0,1])$ satisfy the following properties:
\begin{enumerate}[1)]
\item the restrictions of $f$ and $g$ to $[0,a]$ are $C^2$ contractions, namely the restrictions are $C^2$ diffeomorphisms such that
\[
f(x)<x,\quad g(x)<x\quad\text{for every }x\in (0,a],
\]
\item $f$ and $g$ commute in restriction to $[0,a]$, that is
\[
fg(x)=gf(x)\quad\text{for every }x\in [0,a],
\]
\item the $C^2$ germs of $f$ and $g$ at $0$ generate an abelian free group of rank $2$.
\end{enumerate}
Then for every homeomorphism $\varphi\in \Homeo_+([0,1])$ such that $\varphi f\varphi^{-1}$, $\varphi g\varphi^{-1}$ are $C^2$ in restriction to $[0,\varphi(a)]$, one has that the restriction of $\varphi$ to $(0,a]$ is $C^2$.
\end{prop}

Before giving the proof of the proposition, which encloses the main arguments, we present the main result of the section:

\begin{thm}\label{t:nonC2}
Assume $a\in (0,1)$, $f,g\in \Homeo_+([0,1])$ satisfy the properties of Proposition~\ref{p:nonC2} above. Moreover, assume that there exists $h\in \Homeo_+([0,1])$ such that there exists $t\in (0,a)$ which is a $C^2$ discontinuity point of $h$ and $h(t)\in (0,a)$.

Then the natural action of $\langle f,g,h\rangle \subset \Homeo_+([0,1])$ on $[0,1]$ is not $C^2$-smoothable.
\end{thm}

\begin{proof}[Proof of Theorem~\ref{t:nonC2}]
We argue by contradiction. If there was a homeomorphism $\varphi\in \Homeo_+([0,1])$ such that $\varphi\langle f,g,h\rangle\varphi^{-1}$ is in $\Diff^2_+([0,1])$, then $\varphi$ would satisfy the requirements of Proposition~\ref{p:nonC2}, whence $\varphi$ would be $C^2$ in restriction to $(0,a]$.

However $h$ has a $C^2$ discontinuity point $t\in (0,a)$ and hence the conjugate $\varphi h\varphi^{-1}$ would have $\varphi(t)$ as $C^2$ discontinuity point, against the assumption that $\varphi h\varphi^{-1}$ is $C^2$.
\end{proof}

\begin{proof}[Proof of Proposition~\ref{p:nonC2}]
As $f$ is a $C^2$ contraction when restricted to $[0,a]$, Szekeres' Theorem \ref{t:Szekeres} applies: we denote by $\mathcal X$ the Szekeres vector field of $f$, which is $C^1$, defined on $[0,a)$ and with no singularities on $(0,a)$. We have the assumption 2) that $f$ and $g$ commute in restriction to $[0,a]$, so by Szekeres' theorem $g$ belongs to the Szekeres flow $\{\phi^s_{\mathcal X}\}$. Let $\lambda>0$ be such that $g=\phi^\lambda_{\mathcal X}$. Then by assumption 3), the subgroup $A:=\langle 1,\lambda\rangle\subset \R$ is a dense abelian group of rank $2$.

As $f$ and $g$ are contractions, we have that for any positive power $n\in \N$, the restrictions of the iterates $f^n$ and $g^n$ to the interval $[0,a]$ coincide with the times $\phi^n_{\mathcal X}$ and $\phi^{n\lambda}_{\mathcal{X}}$ respectively (however such a statement is in general not true for negative powers of $f$ and $g$). More generally, we have the following:

\begin{claim}
Denote by $A$ the rank $2$ abelian subgroup of $\R$ generated by $1$ and $\lambda$.
For every $\alpha\in A, \alpha>0$ there exists an element $h_\alpha\in \langle f,g\rangle$ such that
\[
h_\alpha\vert_{[0,a]}(x)=\phi^\alpha_{\mathcal X}(x)\quad \text{for any }x\in [0,a].
\]
Moreover, $f$ and $h_\alpha$ commute on $[0,a]$: $[f,h_\alpha]\vert_{[0,a]}=[h_\alpha,f]\vert_{[0,a]}=id\vert_{[0,a]}$.
\end{claim}

\begin{proof}[Proof of Claim]
Let $\ell,m\in\Z$ be such that $\alpha= \ell+m\lambda$. There exists $y>0$  such that the element $f^\ell g^m$ is equal to $\phi^\alpha_{\mathcal X}$ on the right neighbourhood $[0,y]$. If $y\ge a$, then we set $h_\alpha=f^\ell g^m$ and we are done.

Otherwise, we have $y<a$.
As $f$ is a contraction on $[0,a]$, there exists a positive integer $N\in \N$ such that $f^{N}([0,a])=[0,\phi^N_{\mathcal X}(a)]\subset [0,y]$. Define $h_\alpha=f^{-N}f^\ell g^m f^N$. Then for any $x\in [0,a]$ we have
\begin{align*}
h_\alpha\vert_{[0,a]}(x)=\,&f^{-N}f^\ell g^m f^N\vert_{[0,a]}(x)\\
=\,&f^{-N}f^\ell g^m \vert_{[0,\phi^N_{\mathcal X}(a)]} (\phi^N_{\mathcal X}(x))\\
=\,&f^{-N}\vert_{[0,\phi^{\alpha+N}_{\mathcal X}(a)]}(\phi^{\alpha+N}_{\mathcal X}(x)).
\end{align*}
The element $f^{-N}$ equals $\phi^{-N}_{\mathcal X}$ on the interval $[0,\phi^N_{\mathcal X}(a)]$. Here $\alpha>0$ hence $[0,\phi^{\alpha+N}_{\mathcal X}(a)]\subset [0,\phi^N_{\mathcal X}(a)]$. We conclude that for any $x\in [0,a]$ we have
\[h_\alpha\vert_{[0,a]}(x)=\phi^{-N}_{\mathcal X}\phi^{\alpha}_{\mathcal X}\phi^{N}_{\mathcal X}(x)=\phi^\alpha_{\mathcal X}(x),\]
as desired.
\end{proof}

The claim implies that the group generated by $f$ and $g$ contains a one-parameter flow in its local $C^0$-closure.
Suppose that $\varphi$ is a homeomorphism such that $\varphi \langle f,g,h\rangle\varphi^{-1}$ is in $\Diff_+^2([0,1])$. 

The element $\varphi f\varphi^{-1}$ is a $C^2$ contraction on a right neighbourhood $[0,\varphi(a)]$ of $0$, thus Szekeres' theorem applies again. 
Let $\mathcal{Y}$ denote the Szekeres vector field of $\varphi f\varphi^{-1}$ and let $\{\phi_{\mathcal Y}^s\}$ be the associated one-parameter flow defined on $[0,\varphi(a)]$. 
The elements $\varphi h_\alpha\varphi^{-1}$'s commute with $\varphi f\varphi^{-1}$ on $[0,\varphi(x)]$, hence by Szekeres' theorem we must have 
that their restrictions to $[0,\varphi(a)]$ are contained in the flow $\{\phi_{\mathcal Y}^s\}_{s\ge 0}$. 
Moreover they are \emph{densely} contained because $A$ is dense in $\R$.

\begin{claim}
The restriction of $\varphi$ to $(0,a]$ is $C^1$ and takes the Szekeres vector field $\mathcal X$ of $f$, defined on $[0,a]$, to $\mathcal{Y}$:
\[
\varphi_*\mathcal X=\mathcal Y.
\]
\end{claim}
\begin{proof}[Proof of Claim]
For any $x\in [0,a]$ and $\alpha\in A$, $\alpha>0$, we have
\[
\phi_{\mathcal Y}^\alpha\left (\varphi(x)\right )=\varphi\left (\phi_{\mathcal X}^\alpha(x)\right ).
\]
Now, $\phi_{\mathcal X}^\alpha(x)\neq x$ for any $\alpha>0$, $x\in (0,a]$. Thus for any $x\in (0,a]$ and $\alpha\in A$, $\alpha>0$, we have
\begin{align*}
\frac{\phi_{\mathcal Y}^\alpha\left (\varphi (x)\right )-\varphi(x)}{\alpha}&= \frac{\varphi\left (\phi_{\mathcal X}^\alpha(x)\right )-\varphi(x)}{\alpha} \\
&= \frac{\varphi\left (\phi_{\mathcal X}^\alpha(x)\right )-\varphi(x)}{\phi^\alpha_{\mathcal X}(x)-x}\cdot \frac{\phi^\alpha_{\mathcal{X}}(x)-x}{\alpha}.
\end{align*}
Hence
\[
\frac{\varphi\left (\phi_{\mathcal X}^\alpha(x)\right )-\varphi(x)}{\phi^\alpha_{\mathcal X}(x)-x} = \frac{\phi_{\mathcal Y}^\alpha\left (\varphi (x)\right )-\varphi(x)}{\alpha} \cdot \left(\frac{\phi^\alpha_{\mathcal{X}}(x)-x}{\alpha}\right )^{-1}.
\]
Taking the limit on both sides as $\alpha\in A$, $\alpha>0$, goes to $0$ (recall that $A$ is dense in $\R$), we get on the left hand side the derivative $\varphi'(x)$ and on the right hand side the ratio $\mathcal Y\left (\varphi(x)\right )/\mathcal{X}(x)$ (here
we identify $\mathcal X,\mathcal Y$ to $C^1$ functions). Observe that the ratio $\mathcal Y\left (\varphi(x)\right )/\mathcal{X}(x)$ is well-defined because $\mathcal{X}$ has no singularities on $(0,a)$.

This gives the two desired statements. Indeed, as  $\mathcal Y\left ( \varphi(x)\right )/\mathcal{X}(x)$ is $C^0$ on $(0,a]$, so is $\varphi'$. Hence $\varphi$ is $C^1$ on $(0,a]$. Moreover, taking $\mathcal X(x)$ to the left hand side, we get 
\[
\varphi'(x)\cdot \mathcal{X}(x)=\mathcal Y\left (\varphi(x)\right )\quad\text{for any }x\in [0,a],
\]
that is $\varphi_*\mathcal X=\mathcal Y$, as wanted.
\end{proof}
Now we can conclude the proof. From the previous Claim, we write
\[
\varphi'(x)=\frac{\mathcal Y\left (\varphi(x)\right )}{\mathcal X(x)}\quad\text{for every }x\in (0,a]. 
\]
Moreover the previous Claim gives that the right hand side in the last expression is at least $C^1$ on $(0,a]$ and therefore the same holds for $\varphi'$. This implies that $\varphi$ is $C^2$ on $(0,a]$, as desired.
\end{proof}

\subsection{Thompson-Stein groups}

We finally apply the previous result to prove that the Thompson-Stein groups are not $C^2$-smoothable.
\begin{proof}[Proof of Theorem~\ref{t:F23}]
In the group $F(n_1,\ldots,n_k)$, $k\ge 2$, it is possible to find two elements $f,g$ fixing $0$ such that $f'(0)=1/n_1$, $g'(0)=1/n_2$. Let $a\in (0,1)$ be such that $f,g$ are linear contractions in restriction to $[0,a]$. Consider any element $h$ which is the identity in restriction to $[0,a/2]$ but not in restriction to $[0,a]$. Then there exists $t\in [a/2,a)$ which is a $C^1$ discontinuity point of $h$ with $h(t)\in [a/2,a)$ (actually we may take for $t$ the leftmost point in the support of $h$). Thus we apply Theorem~\ref{t:nonC2} and conclude that the natural action on $[0,1]$ of the group generated by $f,g$ and $h$ is not $C^2$-smoothable. In particular the natural action of $F(n_1,\ldots,n_k)$ on $[0,1]$ is not $C^2$-smoothable.
\end{proof}

\begin{proof}[Proof of Corollary~\ref{t:T23}]
From \cite[Theorem 3.A]{Liousse},  every faithful $C^2$ action of $T(2,n_2,\ldots,n_k)$ on $\T$ is topologically conjugate to its standard piecewise linear action. However $T(2,n_2,\ldots,n_k)$ contains $F(2,n_2,\ldots,n_k)$ as a subgroup, whose standard action on $\T$ cannot be conjugate to a $C^2$ action, after our Theorem~\ref{t:F23}.
Therefore a $C^2$ action of $T(2,n_2,\ldots,n_k)$ cannot be faithful. 

As in \cite[Theorem 3.B']{Liousse}, if we assume furthermore $n_2=3$, the simplicity of $T(2,3,n_3,\ldots,n_k)$ allows to conclude that every $C^2$ action of such a group is trivial.
\end{proof}

\section*{Acknowledgements} 
The authors thank Isabelle Liousse for her valuable comments on the first version of the paper (Corollary~\ref{t:T23} is due to her). \emr{Many thanks also go to the referee for his/her suggestions and for pointing out the many imprecise earlier algebraic statements about $A_\lambda$.}
This work has been done during visits of the authors to the Institut de Math\'ematiques de Bourgogne and the \'Ecole Polytechnique F\'ed\'erale de Lausanne. The authors thank these institutions for the welcoming atmosphere. We are also grateful to the members of the IMB participating to the working seminar ``Dynamique des actions de groupes'' during Autumn 2017, as well as to the organizers of the conference \emph{Dynamics Beyond Uniform Hyperbolicity} held in BYU, Provo, June 2017.
M.T. was partially supported by PEPS -- Jeunes Chercheur-e-s -- 2017 (CNRS).

\end{document}